\documentclass{amsart}

\usepackage{amsmath,amsopn,amssymb,amsthm,mathtools,mathrsfs,yhmath,tensor,geometry,graphicx,stmaryrd,tabulary,booktabs}

\usepackage{xcolor}

\usepackage{accents}
	\newcommand{\ubar}[1]{\underaccent{\bar}{#1}}
	
\usepackage{upgreek,textgreek}
	
\usepackage{hyperref}
	
\usepackage{lipsum} 

\usepackage{parskip}[=v1]
 
\usepackage{tikz}
\usetikzlibrary{calc}

\usepackage{tikz}

\usepackage{pgfplots}
\pgfplotsset{compat=1.17}
\usetikzlibrary{quotes,angles}

\usepackage{booktabs}

\newtheorem*{theorem*}{Theorem}

\newtheorem*{prop*}{Proposition}

\newtheorem*{lemma*}{Lemma}

\newtheorem*{defi*}{Definition}

\newtheorem*{rmk*}{Remark}

\newtheorem*{cor*}{Corollary}

\newtheorem*{claim*}{Claim}

\newtheorem{ftheorem}{Theorem}[section]
\newtheorem*{ftheorem*}{Theorem}
		
\newtheorem{fprop}{Proposition}[section]
\newtheorem*{fprop*}{Proposition}
	
\newtheorem{flemma}{Lemma}[section]
\newtheorem*{flemma*}{Lemma}

\newtheorem{fdefi}{Definition}[section]
\newtheorem*{fdefi*}{Definition}

\newtheorem*{frmk*}{Remark}

\newtheorem{fcor}{Corollary}[section]
\newtheorem*{fcor*}{Corollary}

\newtheorem*{fclaim*}{Claim}

\newtheorem*{fconj*}{Theorem}

\newtheorem{fquest}{Question}[section]
\newtheorem*{fquest*}{Question}

\newcommand{\Address}{{
  \bigskip
  \footnotesize

  Chao-Ming Lin, \textsc{Department of Mathematics, University of California-Irvine, CA}\par\nopagebreak
  \textit{E-mail address:} \texttt{chaominl@uci.edu}\par\nopagebreak
  \textit{Personal Website:} \texttt{https://chaominl.github.io}

}}

\newcommand{\pa}{\partial}

\DeclareMathOperator\osc{osc}
\DeclareMathOperator\Arg{Arg}

\DeclareMathOperator\Tr{Tr}

\DeclareMathOperator\Id{Id}

\DeclareMathOperator\Hess{Hess}

\begin{document}

\title{Deformed Hermitian--Yang--Mills Equation on Compact Hermitian Manifolds}
\author{Chao-Ming Lin}

\begin{abstract}
Let $(X, \omega)$ be a compact connected Hermitian manifold of dimension $n$. We consider the Bott--Chern cohomology and let $[\chi ] \in H^{1,1}_{\text{BC}}(X; \mathbb{R})$. We study the deformed Hermitian--Yang--Mills equation, which is the following nonlinear elliptic equation $\sum_{i} \arctan (\lambda_i) = h(x)$, where $\lambda_i$ are the eigenvalues of $\chi$ with respect to $\omega$. 
\end{abstract}
\maketitle
\vspace{-0.6cm}

\section{Introduction}
Let $X$ be a connected compact complex manifold of dimension $n$ with $\omega$ the Hermitian form on $X$ and $[\chi_0 ] \in H^{1,1}_{\text{BC}}(X; \mathbb{R}) \coloneqq H^{1,1}_{\text{BC}}(X; \mathbb{C}) \cap H^{1,1}(X; \mathbb{R})$. Here, $H^{p, q}_{\text{BC}} (X;\mathbb{C}) \coloneqq \frac{\Omega^{p, q}(X) \cap \ker d}{\sqrt{-1} \partial \bar{\partial} \Omega^{p-1, q-1}(X) }$ is the Bott--Chern cohomology class of $X$.\bigskip

The deformed Hermitian--Yang--Mills equation, which will be abbreviated as dHYM equation later on, was discovered around the same time by Mariño--Minasian--Moore--Strominger \cite{marino2000nonlinear} and Leung--Yau--Zaslow \cite{leung2000special} using different points of view. Mariño--Minasian--Moore--Strominger \cite{marino2000nonlinear} found out that the dHYM equation is the requirement for a $D$-brane on the $B$-model of mirror symmetry to be supersymmetric. It was shown by Leung--Yau--Zaslow \cite{leung2000special} that, in the semi-flat model of mirror symmetry, solutions of the dHYM equation are related via the Fourier--Mukai transform to special Lagrangian submanifolds of the mirror. The study of the dHYM equation for a holomorphic line bundle over a compact Kähler manifold was initiated by Jacob--Yau \cite{jacob2017special}; they introduced the following problem in \cite{jacob2017special}: if $\omega$ is a Kähler form, does there exists a real smooth, closed $(1,1)$ form $\chi \in [\chi_0]$? Such that, 
\begin{align}
\label{eq:1.1}
\Im \left ( \omega + \sqrt{-1} \chi \right )^n &= \tan \left (  \hat{\theta}   \right ) \Re \left ( \omega + \sqrt{-1} \chi \right )^n, \tag{1.1}
\end{align}where $\Im, \Re$ are the imaginary and real parts respectively, and $\hat{\theta}$ is a topological constant determined by $[\omega], [\chi_0]$. In the supercritical phase case, which means that the phase $\hat{\theta}$ satisfies $\hat{\theta} > \frac{n-2}{2} \pi$, the dHYM equation on a compact Kähler manifold was solved by Collins--Jacob--Yau \cite{collins20151} assuming a notion of $C$-subsolution. \bigskip

The main purpose of this work is to solve the dHYM equation on compact Hermitian manifolds. We should emphasize that there are many significant works which have been done on the Kähler case recently. Due to the space limitations, we will only list them here and will only be able to use some of the results therein. The interested reader is referred to \cite{chen2019j, collins2020stability, collins2018moment, collins2018deformed, han2019varepsilon, han2020stability, jacob2019weak, jacob2020deformed, pingali2019deformed, schlitzer2019deformed, takahashi2019collapsing, takahashi2020tan} and the references therein.\bigskip

According to \cite{jacob2017special}, equation (\ref{eq:1.1}) is actually equivalent to the following equation 
\begin{align}
\label{eq:1.2}
\Theta_{\omega}(\chi) \coloneqq \sum_{i = 1}^{n} \arctan (\lambda_i) = \hat{\theta}, \tag{1.2}
\end{align}where $ \lambda_i $ are the eigenvalues of $\omega^{-1} \chi$ and we require that the range of $\hat{\theta}$ has certain constraints which will be discussed later. Equation (\ref{eq:1.2}) is a natural generalization to compact Kähler manifolds of the following special Lagrangian equation
\begin{align*}
\Re \left (  \det \left (  \Id + \sqrt{-1} \Hess F \right )   \right ) = 0
\end{align*}
which is introduced by Harvey--Lawson \cite{harvey1982calibrated}. The special Lagrangian equation is also studied extensively; see, for instance, \cite{caffarelli1985dirichlet, collins2017concavity, nadirashvili2010singular, smoczyk2002mean, smoczyk2004longtime, wang2012mean, wang2013singular, wang2014hessian, yuan2002bernstein, yuan2006global} and the references therein. \bigskip

Here, we would like to ask the following question on a compact Hermitian manifold.

\hypertarget{Q:1.1}{\begin{fquest}} Does there exist a $\chi \in [\chi_0 ] \in H^{1,1}_{\text{BC}}(X; \mathbb{R})$ such that 
\begin{align*}
\label{eq:1.3}
\Theta_{\omega}(\chi) = \sum_{i = 1}^{n} \arctan (\lambda_i) = h(x)? \tag{1.3}
\end{align*}Here $h\colon X \rightarrow \left [ (n-2)\frac{\pi}{2} + \epsilon_0, n \frac{\pi}{2} \right )$ with $\epsilon_0 >0$ and $\{ \lambda_i \}$ are the eigenvalues of $\Lambda$, where $\Lambda^j_k \coloneqq \omega^{j \bar{l}} \chi_{k \bar{l}}$.
\end{fquest}

In this paper, we work under the assumption of a $C$-subsolution, which is introduced in Székelyhidi \cite{szekelyhidi2018fully} (See also Guan \cite{guan2014second}). In \hyperlink{S:3}{section 3}, we first prove the crucial a priori estimates to all orders similar to \cite{jacob2017special}. The difficulty will mainly be the $C^2$ estimates, since we are dealing with Hermitian metric in stead of Kähler metric, thus torsions will pop up. We apply maximum principle to prove the $C^2$ estimate. To simply the proof, we use symmetric functions and perturb the eigenvalues to apply the maximum principle.
\hypertarget{T:1.1}{\begin{ftheorem}[A priori estimates]}
Let $X$ be a connected compact complex manifold of dimension $n$ with $\omega$ the Hermitian form on $X$. Assume $u \colon X \rightarrow \mathbb{R}$ is a smooth function satisfying $\sup_X u = 0$, $[\chi_0 ] \in H^{1,1}_{\text{BC}}(X; \mathbb{R})$, and $\Theta_{\omega}(\chi_0 + \sqrt{-1} \partial \bar{\partial} u) = h(x)$, where $h \colon X \rightarrow [ (n-2)\frac{\pi}{2} + \epsilon_0, n\frac{\pi}{2})$. Suppose that there exists a $C$-subsolution $\ubar{u} \colon X \rightarrow \mathbb{R}$, then for every $\alpha \in (0,1)$, there exists a constant $C= C \left ( X, \omega, \alpha, \epsilon_0, h, \chi_0, \ubar{u}  \right )$ such that 
\begin{align*}
\label{eq:1.4}
\| u \|_{C^{2, \alpha}}\leq C \left ( X, \omega, \alpha, \epsilon_0, h, \chi_0, \ubar{u}   \right ). \tag{1.4}
\end{align*} 
\end{ftheorem}

Then we have the following Existence theorem for a special Hermitian form. Consider the constant $\hat{\Theta}_\omega \left (  \chi \right )$ defined by
\begin{align*}
\label{eq:1.5}
\hat{\Theta}_\omega \left ( \chi  \right ) &\coloneqq \Arg \left (  \int_X \left (  \omega + \sqrt{-1} \chi \right )^n   \right ), \tag{1.5}
\end{align*}where $\chi \in [\chi_0 ] \in H^{1,1}_{\text{BC}}(X; \mathbb{R})$. Here, we need to specify the branch cut.

\hypertarget{T:1.2}{\begin{ftheorem}[Existence theorem]}
Assume the Hermitian form $\omega$ satisfies $\partial \bar{\partial} \omega = 0 = \partial \bar{\partial} \left ( \omega^2 \right )$. Also, suppose that there exists a $C$-subsolution $\ubar{\chi} \coloneqq \chi_0 + \sqrt{-1} \partial \bar{\partial} \ubar{u}$ such that $\Theta_{\omega} \left ( \ubar{\chi} \right ) > (n-2)\frac{\pi}{2}$. Then there exists a unique smooth $(1,1)$-form $\chi \in [\chi_0]$ solving the deformed Hermitian--Yang--Mills equation
\begin{align*}
\label{eq:1.6}
\Theta_\omega \left ( \chi \right ) = \hat{\Theta}_\omega \left ( \ubar{\chi} \right ) = \hat{\Theta}_\omega \left ( \chi_0 \right ). \tag{1.6}
\end{align*}
\end{ftheorem}

The Kähler case was proved by Collins--Jacob--Yau \cite{collins20151} with the requirements of $C$-subsoution and supercritical phase condition. Chen \cite{chen2019j} also proved the existence theorem under a numerical condition and when the phase is more restrictive. \bigskip

To study Hermitian manifolds, we consider the condition $\partial \bar{\partial} \left (  \omega^k \right ) = 0$, which plays an important role. For example, if $k = 1$, that is, when a Hermitian form satisfies $\partial \bar{\partial} \omega = 0$, we call the form pluriclosed. Streets--Tian \cite{streets2010parabolic, streets2012generalized, streets2013regularity} studied a parabolic flow for pluriclosed forms, which is called pluriclosed flow. Tosatti--Weinkov \cite{tosatti2015evolution} and Tosatti--Weinkov--Yang \cite{tosatti2015collapsing} studied the Chern--Ricci flow on compact complex surfaces starting from a pluriclosed form. See also \cite{arroyo2019long, boling2016homogeneous, enrietti2015pluriclosed, jordan2020calabi, verbitsky2014rational, zhao2019strominger} and the references therein. If $k = n-1$, that is, when a Hermitian form satisfies $\partial \bar{\partial}  \left (\omega^{n-1} \right ) = 0$, then the corresponding Hermitian metric is called Gauduchon. Gauduchon \cite{gauduchon1977theoreme} proved that for a given Hermitian metric, there exists a Gauduchon metric in the same conformal class; the Gauduchon metric is unique up to a constant factor. See also \cite{li1987hermitian, lubke1995kobayashi} and the references therein. If $k = n-2$, that is, when a Hermitian form satisfies $\partial \bar{\partial}  \left (\omega^{n-2} \right ) = 0$, the form is called astheno-Kähler, which was introduced by Jost--Yau \cite{jost1993nonlinear} to establish the existence of Hermitian Harmonic maps. The astheno-Kähler condition turns out to be particularly interesting for many analytic arguments to be useful. For example, Tosatti--Weinkov \cite{tosatti2019hermitian} proved the Calabi--Yau theorems for Gauduchon and strongly Gauduchon metrics on the class of compact astheno-Kähler manifolds. Phong--Picard--Zhang \cite{phong2016anomaly, phong2017anomaly, phong2018anomaly, phong2018geometric, phong2019flow} studied the anomaly flow on a compact complex manifold, which admits a non-vanishing holomorphic $(n,0)$-form $\Omega$ and whose stationary points are astheno-Kähler metrics. See also \cite{fino2019astheno, fino2011astheno, latorre2017non, matsuo2009astheno, matsuo2001compact} and the references therein.\bigskip

The condition $\partial \bar{\partial} \left ( \omega^k \right ) = 0$ for all $1 \leq k \leq n-1$ was studied by Guan--Li \cite{guan2009complex} when proving the complex Monge--Ampère equation on closed complex $n$-dimensional Hermitian manifolds. The purpose is to carry out the continuity method. Our condition $\partial \bar{\partial} \omega = 0 = \partial \bar{\partial} \left ( \omega^2 \right )$ here is again a technical condition; we use this condition to make sure that the constant $\hat{\Theta}_\omega \left ( \chi \right )$ is fixed under the same cohomology class $[\chi]$. Then, we can construct a continuity path to apply the a priori estimates and solve equation (\ref{eq:1.6}); the details are in \hyperlink{S:4}{section 4}. Notice that the condition $\partial \bar{\partial} \left ( \omega^k \right ) = 0$ for $k =1, 2$ is sufficient to prove $\partial \bar{\partial} \left ( \omega^k \right ) = 0$ for all $1 \leq k \leq n-1$. See, for example, \cite{fino2011astheno}.\bigskip

There are Hermitian forms satisfying the condition $\partial \bar{\partial} \omega = 0 = \partial \bar{\partial}   ( \omega^2   )$. For example, on a compact complex surface, by Gauduchon \cite{gauduchon1977theoreme}, any Hermitian metric will have a Gauduchon metric in the same conformal class, which implies that the Gauduchon metric is pluriclosed. Moreover, since the manifold is a surface, the second equality will always hold. This observation will give us more examples. For instance, let $M$ be a Kähler manifold and $N$ be a compact complex surface, then $M \times N$ will also have a Hermitian form $\omega$ satisfying $\partial \bar{\partial} \omega = 0 = \partial \bar{\partial}  ( \omega^2   )$. The Hermitian form comes from pulling back the Kähler form on the Kähler manifold $M$ plus pulling back the Gauduchon metric on the compact complex surface $N$. Now, let us focus on the compact complex surface, as a corollary of \hyperlink{T:1.2}{the existence theorem}. We have the following.

\hypertarget{C:1.1}{\begin{fcor}}
Let $X$ be a compact complex surface equipped with a Hermitian metric $\omega$. Suppose that there exists a $C$-subsolution $\ubar{\chi} \coloneqq \chi_0 + \sqrt{-1} \partial \bar{\partial} \ubar{u}$ such that $\Theta_{\omega} \left ( \ubar{\chi} \right ) > 0$. Then there exists a pluriclosed Hermitian metric $\tilde{\omega}$ in the conformal class of $\omega$ and a unique smooth $(1,1)$-form $\chi \in [\chi_0]$ solving the deformed Hermitian--Yang--Mills equation
\begin{align*}
\label{eq:1.7}
\Theta_{\tilde{\omega}} \left ( \chi \right ) = \hat{\Theta}_{\tilde{\omega}} \left ( \ubar{\chi} \right ) = \hat{\Theta}_{\tilde{\omega}} \left ( \chi_0 \right ). \tag{1.7}
\end{align*}
\end{fcor}\bigskip

Moreover, we solve the dHYM equation when the non-Kähler compact complex surface is either an {\em Inoue surface} or a {\em secondary Kodaira} surface.

\hypertarget{C:1.2}{\begin{fcor}}
Let $X$ be either a Inoue surface or a secondary Kodaira surface with a Hermitian metric $\omega$. Then for any $[\chi_0] \in H^{1,1}_{\text{BC}}(X; \mathbb{R})$, there exists a pluriclosed Hermitian metric $\tilde{\omega}$ in the same conformal class of $\omega$ and a unique smooth $(1,1)$-form $\chi \in [\chi_0]$ solving the deformed Hermitian--Yang--Mills equation
\begin{align*}
\Theta_{\tilde{\omega}} \left ( \chi \right ) = \hat{\Theta}_{\tilde{\omega}} \left ( \chi \right ).
\end{align*} 
\end{fcor}

The paper is organized as follows. In \hyperlink{S:2}{section 2}, we list some basic lemmas and state the definition of $C$-subsolution. In \hyperlink{S:3}{section 3}, we prove the a priori estimate \hyperlink{T:1.1}{Theorem 1.1}. More precisely, we mainly focus on proving the $C^2$ estimate; in addition, we also get a direct $C^1$ estimate. In \hyperlink{S:4}{section 4}, we prove \hyperlink{T:1.2}{the existence theorem} for a special type  Hermitian metric, the existence theorem on a compact complex surface, and in particular when the non-Kähler surface is either a {\em Inoue surface} or a {\em secondary Kodaira} surface.\bigskip

{\em Note:} As we are about to submit our preprint, we notice that there is an interesting paper by Huang--Zhang--Zhang \cite{huang2020almosthermitian} appeared recently on arXiv:2011.14091v1 working on a similar problem. After a brief review, we notice the following differences. Huang--Zhang--Zhang work on a more general manifold, which is almost Hermitian manifolds, but their phase condition is more restrictive---they assume the hypercritical phase condition, which is $h(x) > \frac{(n-1)\pi}{2}$, to make the equation concave. But here we only need to assume the supercritical phase condition, which is $h(x) > \frac{(n-2)\pi}{2}$. Also, their existence theorem needs another assumption on the existence of a supersolution. Our existence results (\hyperlink{T:1.2}{Theorem 1.2}) is obtained on a large family of Hermitian manifolds, which includes the {\em Inoue Surfaces} and the {\em Secondary Kodaira Surfaces} as particular interesting examples.\bigskip

\textbf{Acknowledgements.} I thank my advisors Prof.~Zhiqin Lu and Prof.~Xiangwen Zhang for hosting many seminars and giving me enlightening helps. It would have been impossible for me to finish this paper without them. I also like to thank Prof.~Tristan Collins for his comments and suggestions on the first version of this paper. 

\hypertarget{S:2}{\section{Preliminaries}}
\subsection{Basic Formulas of Eigenvalues and Symmetric Functions}

Let us state some lemmas for symmetric functions. One can also check the following references \cite{szekelyhidi2018fully, szekelyhidi2017gauduchon} for more details.

\hypertarget{L:2.1}{\begin{flemma}}
Let $\Lambda$ be a diagonal matrix with distinct eigenvalues $\lambda_i$, then the partial derivatives of the eigenvalues $\lambda_i$ with respect to the entries at the diagonal matrix $\Lambda$ are 
\begin{align*}
\label{eq:2.1}
\lambda_i^{pq} &= \frac{\partial \lambda_i}{\partial \Lambda_{p \bar{q}}} = \delta_{pi} \delta_{qi}, \tag{2.1} \\
\label{eq:2.2}
\lambda_{i}^{pq,rs} &= \frac{\partial^2 \lambda_i}{\partial \Lambda_{p \bar{q}} \partial \Lambda_{r \bar{s}}} = (1- \delta_{ip}) \frac{\delta_{iq} \delta_{ir} \delta_{ps}}{\lambda_i - \lambda_{p}} + (1- \delta_{ir}) \frac{\delta_{is} \delta_{ip} \delta_{rq}}{\lambda_i - \lambda_r}. \tag{2.2}
\end{align*}
\end{flemma}

\hypertarget{L:2.2}{\begin{flemma}}
If $F(\Lambda) = f(\lambda_1, \dots, \lambda_n)$ is a function in terms of the eigenvalues $\lambda$ of a diagonalizable matrix $\Lambda$, then at a diagonal matrix $\Lambda$ with distinct eigenvalues $\lambda_i$, we get
\begin{align*}
\label{eq:2.3}
F^{ij} &= \frac{\partial F}{\partial \Lambda_{i \bar{j}}} = \delta_{ij} f_i, \tag{2.3}\\ 
\label{eq:2.4}
 F^{ij,rs} &= \frac{\partial^2 F}{\partial \Lambda_{i \bar{j}} \Lambda_{r \bar{s}}} = f_{ir} \delta_{ij} \delta_{rs} + \frac{f_i - f_j}{ \lambda_i - \lambda_j} (1 - \delta_{ij})\delta_{is} \delta_{jr},\tag{2.4}
\end{align*}where $f_i = \frac{\partial f}{\partial \lambda_i}$.
\end{flemma}

\hyperlink{L:2.2}{Lemma 2.2} here will be used very often in the next section. Roughly speaking, by perturbing the eigenvalues, we can assume that the eigenvalues are distinct, then we can apply \hyperlink{L:2.2}{Lemma 2.2} to simplify our proof of a priori estimate.

\subsection{\texorpdfstring{$C$}{C}-Subsolution}
Here, we review the notion of a $C$-subsolution, which is introduced in Székelyhidi \cite{szekelyhidi2018fully} (see also Guan \cite{guan2014second}).

\hypertarget{D:2.1}{\begin{fdefi}[\cite{szekelyhidi2018fully}]} Fix $\chi_0 \in [\chi]$. We say that a smooth function $\ubar{u} \colon X \rightarrow \mathbb{R}$ is a $C$-subsolution of equation (\ref{eq:1.3}) if the following holds: At each point $x \in X$ define the matrix $\Lambda^i_j \coloneqq \omega^{i \bar{k}} \left (  \chi_0 + \sqrt{-1} \partial \bar{\partial} \ubar{u}   \right )_{j \bar{k}}$. Then we require that the set
\begin{align*}
\label{eq:2.5}
\left \{  \lambda' \in \Gamma : \sum_{l = 1}^{n} \arctan(\lambda'_l) = h(x), \text{ and } \lambda' - \lambda(\Lambda(x)) \in \Gamma_n  \right \} \tag{2.5}
\end{align*}is bounded, where $\lambda(\Lambda(x)) = \left ( \lambda_1, \dots, \lambda_n \right )$ denotes the $n$-tuple of eigenvalues of $\Lambda(x)$ with $\lambda_1 \geq \lambda_2 \geq \dots \geq \lambda_n$.
\end{fdefi}

\begin{figure}[!h] 
\centering 
\begin{minipage}[t]{0.4\textwidth} 
\centering 
\begin{tikzpicture}
\begin{axis}[axis on top,axis lines=middle,xlabel={$\lambda'_1$},ylabel={$\lambda'_2$},
    width=0.9\textwidth,
    trig format plots=rad,
    samples=101,
    unbounded coords=jump,
    xmin=-pi/2,xmax=pi,
    ymin=-pi/2,ymax=pi,
    xtick={-1},xticklabels={},
    ytick={-1},yticklabels={},
    grid=major,grid style={densely dashed},
    legend style={at={(0.5,0.99)},anchor=north west}
    ]
\fill [green] (1/8,1/8) rectangle (pi/2+1,pi/2+1);
\fill [red!50] (1/8,1/8) circle (1pt) node[black,below right] {$(\lambda_1, \lambda_2)$};
\addplot[blue,fill=white,thick,variable=\t,domain=-pi/2+1:pi/2-0.1] ({tan(\t)},{tan(pi/4 -\t)});
\addlegendentry{$\sum_i \arctan(\lambda'_i)  = \sigma$}
\addplot[black,line width=0.5mm,variable=\t,domain={atan(1/8)}:{pi/4-atan(1/8)}] ({tan(\t)},{tan(pi/4 -\t)});
\end{axis}
\end{tikzpicture}
\caption{A $C$-Subsolution.}
\end{minipage} 
\hspace{0.1\textwidth} 
\begin{minipage}[t]{0.4\textwidth} 
\centering 
\begin{tikzpicture}
\begin{axis}[axis on top,axis lines=middle,xlabel={$\lambda'_1$},ylabel={$\lambda'_2$},
    width=0.9\textwidth,
    trig format plots=rad,
    samples=101,
    unbounded coords=jump,
    xmin=-pi/2,xmax=pi,
    ymin=-pi/2,ymax=pi,
    xtick={-1},xticklabels={},
    ytick={-1},yticklabels={},
    grid=major,grid style={densely dashed},
    legend style={at={(0.5,0.99)},anchor=north west}
    ]
\fill [red!50] (-1,-1) rectangle (pi/2+1.5,pi/2+1.5);
\addplot[blue,thick,variable=\t,domain=-pi/2+1:pi/2-0.1] ({tan(\t)},{tan(pi/4 -\t)});
\addlegendentry{$\sum_i \arctan(\lambda'_i) = \sigma$}
\end{axis}
\end{tikzpicture}
\caption{$C$-Subsolutions.}
\end{minipage} 
\end{figure} 

Here, note that when a $C$-subsolution $\ubar{u}$ exists, for convenience, we say $\ubar{\chi} \coloneqq \chi_0 + \sqrt{-1} \partial \bar{\partial}  \ubar{u}$ as a $C$-subsolution. The above figures are examples of a $C$-subsolution at a point and $C$-subsolutions at a point when $n = 2$.\bigskip

We have the following Lemma due to Wang--Yuan \cite{wang2014hessian} and Yuan \cite{yuan2006global}.
\hypertarget{L:2.3}{\begin{flemma}[\cite{collins20151}, \cite{wang2014hessian} and \cite{yuan2006global}]}
Suppose we have real numbers $\lambda_1 \geq \lambda_2 \geq \dots \geq \lambda_n$ satisfying $\Theta(\lambda) = \sigma$ for $\sigma \geq (n-2) \frac{\pi}{2} + \epsilon_0$, where $\Theta(\lambda) \coloneqq \sum_{i = 1}^n \arctan(\lambda_i)$ and $\epsilon_0 \geq 0$. Then $(\lambda_1, \dots , \lambda_n)$ have the following arithmetic properties:
\begin{itemize}
\item[(i)] $\lambda_{n-1} + \lambda_n \geq \tan \left ( \frac{\epsilon_0}{2} \right )$.
\item[(ii)] $\sigma_k (\lambda_1, \dots, \lambda_n) \geq 0$ for all $1 \leq k \leq n-1$.
\end{itemize}Furthermore,
\begin{itemize}
\item[(iii)] If $\Gamma^{\sigma}$ is not empty, the boundary $\partial \Gamma^{\sigma}$ is smooth, convex hypersurface.
\end{itemize}In addition, if $\sigma \geq (n-2)\frac{\pi}{2} + \epsilon_0$, then
\begin{itemize}
\item[(iv)] If $\lambda_n \leq 0$, then $|\lambda_n| \leq C(\epsilon_0)$.
\end{itemize}
\end{flemma}

Furthermore, we have the following explicit description of the $C$-subsolution due to Collins--Jacob--Yau \cite{collins20151}.

\hypertarget{L:2.4}{\begin{flemma}[\cite{collins20151}]}
A smooth function $\ubar{u} \colon X \rightarrow \mathbb{R}$ is a $C$-subsolution of equation (\ref{eq:1.3}) iff at each point $x \in X$,  for all $j = 1, \dots, n$ we have
\begin{align*}
\sum_{l \neq j} \arctan \left ( \mu_l \right ) > h(x) - \frac{\pi}{2},
\end{align*}where $\left \{  \mu_1, \mu_2, \dots, \mu_n \right \}$ are the eigenvalues of the Hermitian endomorphism $\omega^{i \bar{k}} \left (  \chi_0 + \sqrt{-1} \partial \bar{\partial} \ubar{u}   \right )_{j \bar{k}}$.
\end{flemma}

\begin{figure}[ht]
\centering
\begin{tikzpicture}
\begin{axis}[axis on top,axis lines=middle,xlabel={$\lambda'_1$},ylabel={$\lambda'_2$},
    width=0.4\textwidth,
    trig format plots=rad,
    samples=101,
    unbounded coords=jump,
    xmin=-pi/2,xmax=pi,
    ymin=-pi/2,ymax=pi,
    xtick={-1},xticklabels={},
    ytick={-1},yticklabels={},
    grid=major,grid style={densely dashed},
    legend style={at={(0.5,0.99)},anchor=north west}
    ]
\fill [red!25] (-pi/2,-1) rectangle (pi/2+1.5,pi/2+1.5);
\fill [red!25] (-1,-pi/2) rectangle (pi/2+1.5,pi/2+1.5);
\fill [red!50] (-1,-1) rectangle (pi/2+1.5,pi/2+1.5);
\addplot[blue,thick,variable=\t,domain=-pi/2+1:pi/2-0.1] ({tan(\t)},{tan(pi/4 -\t)});
\addlegendentry{$\sum_i \arctan(\lambda'_i)  = \sigma$}
\end{axis}
\end{tikzpicture}
\caption{Region of $C$-Subsolutions.}
\end{figure}
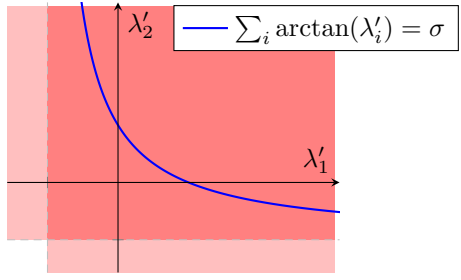
The above figure is an example to show the equivalence at a point when $n = 2$.\bigskip

Note that by the definition and the fact that the manifold is compact, we have the following two observations:
(i) Having $\ubar{\chi}$ as a $C$-subsolution of $h$ is an open condition.
(ii) Those who have $\ubar{\chi}$ as a $C$-subsolution are closed under maximum and minimum.

\hypertarget{L:2.5}{\begin{flemma}}
If $\ubar{\chi}$ is a $C$-subsolution of the equation $\Theta_{\omega}(\chi) = \sum_{i = 1}^{n} \arctan (\lambda_i) = h(x)$, where $h \colon X \rightarrow \left ( (n-2)\frac{\pi}{2}, n \frac{\pi}{2} \right )$ is a continuous function. 
\begin{itemize}
\item[(i)] There exists an $\epsilon > 0$ such that for any continuous function $k\colon X \rightarrow \left ( (n-2)\frac{\pi}{2}, n \frac{\pi}{2} \right )$ with $|h - k|_{C^0} < \epsilon$, $\ubar{\chi}$ is again a $C$-subsolution of $k(x)$.
\item[(ii)] If $\ubar{\chi}$ is a $C$-subsolution of $h_1$ and $h_2$, then $\ubar{\chi}$ is also a $C$-subsolution of $\max \left \{ h_1, h_2 \right \}$ and $\min \left \{ h_1, h_2 \right \}$.
\end{itemize}
\end{flemma}

\begin{proof}
For (i), assume $\ubar{\chi}$ is a $C$-subsolution of the equation 
\begin{align*}
\Theta_{\omega}(\chi) = \sum_{i = 1}^{n} \arctan (\lambda_i) = h(x), 
\end{align*}where $h\colon X \rightarrow \left ( (n-2)\frac{\pi}{2}, n \frac{\pi}{2} \right )$ is a continuous function. Let $\left \{  \mu_1, \mu_2, \dots, \mu_n \right \}$ be the eigenvalues of the Hermitian endomorphism $\omega^{i \bar{k}} \left (  \chi_0 + \sqrt{-1} \partial \bar{\partial} \ubar{u}   \right )_{j \bar{k}} = \omega^{i \bar{k}} {\ubar{\chi}}_{j \bar{k}}$. Then by \hyperlink{L:2.4}{Lemma 2.4}, we have
\begin{align*}
\sum_{l \neq j} \arctan \left ( \mu_l \right ) > h(x) - \frac{\pi}{2},
\end{align*}for all $j = 1, \dots, n$. Consider the following continuous function 
\begin{align*}
H(x) \coloneqq \Theta_\omega \left ( \ubar{\chi} \right ) - \mu_{\max}(x) - h(x),
\end{align*}where $\mu_{\max}(x)$ means the largest eigenvalue of $\{ \mu_1, \dots, \mu_n \}$ at $x$. We can see that $H(x) > -\frac{\pi}{2}$ for all $x \in X$. Since $X$ is compact and $H$ is continuous, we can find $\epsilon > 0$ sufficiently small, such that
\begin{align*}
H(x) > -\frac{\pi}{2} +  \epsilon,
\end{align*}for all $x \in X$. Thus, for any continuous function $k\colon X \rightarrow \left ( (n-2)\frac{\pi}{2}, n \frac{\pi}{2} \right )$ with $|h - k|_{C^0} < \epsilon$, we have
\begin{align*}
\sum_{l \neq j} \arctan \left ( \mu_l \right ) \geq H(x) + h(x) > H(x) + k(x) - \epsilon > k(x) - \frac{\pi}{2}.
\end{align*}By \hyperlink{L:2.4}{Lemma 2.4}, this implies that $\ubar{\chi}$ is also a $C$-subsolution of $k(x)$.\bigskip

For (ii), it follows directly from \hyperlink{L:2.4}{Lemma 2.4}.
\end{proof}

Let us state the following proposition to end this section. The proposition is due to Székelyhidi \cite{szekelyhidi2018fully}, refining previous work of Guan \cite{guan2014second}.
\hypertarget{P:2.1}{\begin{fprop}[\cite{szekelyhidi2018fully}, \cite{guan2014second}]}
Let $[a,b] \subset \left ( \left (  n-2   \right ) \frac{\pi}{2}, n \frac{\pi}{2} \right )$ and $\delta, R > 0$. There exists $\kappa > 0$ with the following property: Suppose that $\sigma \in [a,b]$ and $B$ is a hermitian matrix such that
\begin{align*}
\label{eq:2.6}
\left (  \lambda (B)  - 2 \delta \Id  + \Gamma_n \right ) \cap \partial \Gamma^{\sigma} \subset B_R(0). \tag{2.6}
\end{align*}Then for any hermitian matrix $A$ with $\lambda(A) \in \partial \Gamma^\sigma$ and $|\lambda(A)| > R$, we either have 
\begin{align*}
\label{eq:2.7}
\sum_{p, q} \eta^{p \bar{q}}(A) \left [  B_{p \bar{q}} - A_{p \bar{q}}   \right ] > \kappa \sum_{p} \eta^{p \bar{p}} (A) \tag{2.7}
\end{align*}or $\eta^{i \bar{i}} (A) > \kappa \sum_p \eta^{p \bar{p}}(A)$ for all $i$, where $\eta = \Id + A^2$.
\end{fprop}

\hypertarget{S:3}{\section{A Priori Estimate}}
\subsection{A Priori Estimate}

First, let us summarize the proof of a priori estimate. Under the assumption of $C$-subsolution, by the Alexandroff--Bakelman--Pucci estimate, the $C^0$ estimate is due to Székelyhidi \cite{szekelyhidi2018fully}. Then, we use the maximum principle to obtain that the $C^2$ norm can be bounded by the $C^1$ norm. The Kähler case is due to Collins--Jacob--Yau \cite{collins20151}. Here, different from their work, we deal with the Hermitian form, and thus need to consider the torsions. By the equation itself and the phase condition assumption, the maximum principle will not encounter any blow up situation. Once we have the above type inequality, by a blowup argument due to Dinew--Kołodziej \cite{dinew2017liouville}, we can get the $C^1$ estimate. In the last of this section, we also obtain a direct $C^1$ estimate thanks to Collins's advices, which comes from Collins--Yau \cite{collins2018moment}. \bigskip

To get further $C^{2, \alpha}$ estimate, we look at equation (\ref{eq:1.3}) itself, equation (\ref{eq:1.3}) is only concave when $h \colon X \rightarrow [(n-1)\frac{\pi}{2}, n \frac{\pi}{2})$. However, by \hyperlink{L:2.3}{Lemma 2.3, (i)}, we know that in the supercritical phase case, the level set is convex. In order to apply the standard Evans--Krylov theory to get the $C^{2, \alpha}$ estimate, Collins--Jacob--Yau applied an extension trick, which is due to Wang \cite{wang2012c} (see also Tosatti--Wang--Weinkov--Yang \cite{tosatti2015c}). After the extension, they get a real uniformly elliptic concave operator, to which they apply the Evans--Krylov theory. Alternatively, to get the $C^{2, \alpha}$ estimate, we can also follow Collins--Picard--Wu \cite{collins2017concavity}. They introduced an elliptic operator which is indeed concave, and whose level sets agree with the level sets of the original operator. Last, for the higher order estimates, we can get it by standard Schauder theory. \bigskip

The crucial part for this section will be the $C^2$ estimate while other estimates could be obtained as before. Here, we use a different test function and a perturbation trick to apply the maximum principle.

\subsection{\texorpdfstring{$C^2$}{C2} Estimate of Deformed Hermitian--Yang--Mills Equation}
In this section, we prove the following $C^2$ estimate, which is a generalization of Collins--Jacob--Yau's result \cite{collins20151}. They proved this theorem when $\omega$ is Kähler and $[\chi_0] \in H^{1,1}(X; \mathbb{R})$, we consider the case that $\omega$ is Hermitian and $[\chi_0]$ is in the Bott--Chern cohomology instead. We use symmetric functions to simplify their proof and also we find a way to generalize it to the Hermitian case. To deal with the Hermitian case, we need to handle the torsions. One can see that the dHYM equation is actually a nice equation in the sense that when dealing with some singularities coming from the torsions, the signs which are obtained from the dHYM equation itself actually work pretty well. We will see that later.

\hypertarget{T:3.1}{\begin{ftheorem}[$C^2$-estimates]}
Let $X^n$ be a connected compact Hermitian manifold with Hermitian form $\omega$. Suppose $u \colon X \rightarrow \mathbb{R}$ is a smooth function solving $\Theta_{\omega}(\ubar{\chi} + \sqrt{-1} \partial \bar{\partial} u) = h(x)$, where $h \colon  X \rightarrow [(n-2)\frac{\pi}{2} + \epsilon_0, n \frac{\pi}{2} )$ for some $\epsilon_0 > 0$ and $\ubar{\chi}$ is the $C$-subsolution. Then there exists a constant $C$ such that 
\begin{align*}
\label{eq:3.1}
|\partial \bar{\partial} u |_{\omega} \leq C \left ( 1 + \sup_X |\nabla u|^2_{\omega} \right), \tag{3.1}
\end{align*}where $C = C\left ( |h|_{C^2(X, \omega)}, \osc_X u, \ubar{\chi}, \omega, \epsilon_0  \right)$ and $\nabla$ is the Chern connection with respect to $\omega$.
\end{ftheorem}

To start with, let $\Lambda^j_i \coloneqq   \omega^{j \bar{p}} \left ( \ubar{\chi} + \sqrt{-1} \partial \bar{\partial} u \right )_{i \bar{p}} = \omega^{j \bar{p}} \hat{\chi}_{i \bar{p}}$ be a Hermitian endomorphism, where $\hat{\chi} \coloneqq \ubar{\chi} + \sqrt{-1} \partial \bar{\partial} u$, and let $\lambda_1 \geq \lambda_2 \geq \dots \geq \lambda_n$ be the eigenvalues of $\Lambda$. Note that $\ubar{\chi}$ here is a $C$-subsolution.\bigskip

When $\omega$ is a Hermitian metric, we have to deal with torsions, which is due to that $\omega$ is not Kähler. We will show that the dHYM equation is good in the sense that the torsions can be handled well. Moreover, similar to \cite{collins20151}, we have to choose a nice coordinate near a particular point. Here, since it is not Kähler, the best we can do is the following coordinate, which is used by Guan--Li \cite{guan2009complex} and Streets--Tian \cite{streets2013regularity},

\hypertarget{L:3.1}{\begin{flemma}[\cite{guan2009complex}, \cite{streets2013regularity}]}At any point $x \in X$, there exists a holomorphic coordinate at $x$ such that 
\begin{align*}
\label{eq:3.2}
\omega_{i \bar{j}} (x) = \delta_{ij}; \quad \hat{\chi}_{i \bar{j}} (x) = \lambda_i \delta_{ij};  \quad \omega_{i \bar{i}, j} (x) = 0,  \tag{3.2} 
\end{align*}for all $i, j \in \{ 1, \dots, n \}$.
\end{flemma}

\begin{proof}
Let $\left \{ z_1, z_2, \dots, z_n   \right \}$ be a local coordinate near $x$ such that $z_i (x) = 0, \omega_{i \bar{j}}(x) = \delta_{ij}$, and $\hat{\chi}_{i \bar{j}}(x) = \lambda_i \delta_{ij}$.

We can define a new coordinate $\left \{ w_1, w_2, \dots, w_n \right \}$ by $w_i \coloneqq z_i + \sum_{j \neq i} \frac{\partial \omega_{i \bar{i}}}{\partial z_j}(x) z_i z_j + \frac{1}{2} \frac{\partial \omega_{i \bar{i}}}{\partial z_i}(x) z_i^2$.\\
One can check that $\omega \left ( \frac{\partial}{\partial w_i} , \frac{\partial}{\partial \bar{w}_j}  \right )(x) = \delta_{ij}; \quad \hat{\chi} \left (  \frac{\partial}{\partial w_i} , \frac{\partial}{\partial \bar{w}_j} \right ) (x) = \lambda_i \delta_{ij}; \quad \frac{ \partial \omega \left (  \frac{\partial}{\partial w_i} , \frac{\partial}{\partial \bar{w}_i}   \right )}{\partial w_j} (x) = 0$.

\end{proof}

From now on, without further notice, we always use the above coordinate. Let $\hat{\chi} \coloneqq \ubar{\chi} + \sqrt{-1} \partial \bar{\partial} u$. We can define a Hermitian metric $\eta_{i \bar{j}} \coloneqq \omega_{i \bar{j}} + \hat{\chi}_{i \bar{q}} \omega^{p \bar{q}} \hat{\chi}_{p \bar{j}}$. With this Hermitian metric, we get the following elliptic operator, which actually comes from the linearization of the operator $\Theta_{\omega}(\hat{\chi})$,
\begin{align*}
\label{eq:3.3}
 \Delta_{\eta} \coloneqq  \eta^{i \bar{j}}  \frac{\pa^2}{\pa z_i\pa\bar z_j}. \tag{3.3}
\end{align*}

Now, slightly different from Collins--Jacob--Yau \cite{collins20151}, we consider the following test function
\begin{align*}
\label{eq:3.4}
U_0 (x) \coloneqq \log \left (  C_{\epsilon_0} + \lambda_{\text{max}}   \right ) - \frac{1}{2}  \log \left( 1 - \frac{| \nabla u |^2}{2 K}  \right )  - A u  \tag{3.4}
\end{align*}where $C_{\epsilon_0} \geq 2 C(\frac{\epsilon_0}{2})$, $K \coloneqq 1 + \sup_X \left | \nabla u \right |^2_\omega$ and $A$ is a constant which will be determined later. Here $C(\frac{\epsilon_0}{2})$ is the constant in \hyperlink{L:2.3}{Lemma 2.3}.\bigskip 

The above quantity is inspired by Hou--Ma--Wu \cite{hou2010second} for the complex Hessian equations and subsequently used by Székelyhidi \cite{szekelyhidi2018fully} for a larger class of concave equations. One can also check \cite{chou2001variational, collins20151, hou2010second, szekelyhidi2018fully, wang2014hessian, smoczyk2002mean} for similar test function. We want to apply the maximum principle to the above test function $U_0$. Let us say that the test function $U_0 (x)$ achieves the maximum at $x_0$. To apply the maximum principle, we need to show the function is twice differentiable at the maximum point $x_0$. But the function $U_0 (x)$ is just a continuous function---it might not be twice differentiable at the maximum point, so here we apply a perturbation trick similar to the one used in \cite{collins20151} and \cite{szekelyhidi2018fully}. Here, we choose a constant diagonal matrix $B$, defined near $x_0$. The choice of $B$ depends on the smallest eigenvalue $\lambda_n$ of $\Lambda$ at the maximum point $x_0$.
  
\begin{itemize}
\hypertarget{tilde Lambda}{\item} If $\lambda_n > 0$, then we pick the constant matrix $B$ to be a diagonal matrix with real entries $B_{11} = 0$ and 
\begin{align*}
0 < B_{22} < \dots < B_{n-1 \ n-1}  < \frac{1}{2} \min \left \{ \kappa, \tan \left (  \frac{\epsilon_0}{2}  \right )  \right \},\quad 0 < B_{nn} < \frac{1}{2} \min \left \{ \kappa, \tan \left (  \frac{\epsilon_0}{2}  \right )  \right \}
\end{align*}
such that $\lambda_1 = \tilde{\lambda}_1 > \tilde{\lambda}_2 > \dots > \tilde{\lambda}_{n-1} > \frac{\tan \left (  \frac{\epsilon_0}{2}  \right )}{2} > \tilde{\lambda}_n > 0$ and $\sum_i \arctan \left ( \tilde{\lambda}_i   \right ) > \frac{n-2}{2} \pi + \frac{\epsilon_0}{2}$.
\item If $\lambda_n \leq 0$, then we pick the constant matrix $B$ to be a diagonal matrix with real entries $B_{11} = 0 = B_{nn}$ and 
\begin{align*}
0 < B_{22} < \dots < B_{n-1 \ n-1}  < \frac{1}{2} \min \left \{ \kappa, \tan \left (  \frac{\epsilon_0}{2}  \right )  \right \}
\end{align*}
such that $\lambda_1 = \tilde{\lambda}_1 > \tilde{\lambda}_2 > \dots > \tilde{\lambda}_{n-1} > \frac{\tan \left (  \frac{\epsilon_0}{2}  \right )}{2} > 0 \geq \tilde{\lambda}_n = \lambda_n$ and $\sum_i \arctan \left ( \tilde{\lambda}_i   \right ) > \frac{n-2}{2} \pi + \frac{\epsilon_0}{2}$.\bigskip
\end{itemize}

Note that the $\kappa$ here is the quantity in \hyperlink{P:2.1}{Proposition 2.1}. For either case, we want all the eigenvalues of $\tilde{\Lambda} = \Lambda - B$ are distinct from each other at the maximum point $x_0$, hence they are again distinct from each other near $x_0$. We will show that in the end the estimate will not depend on the choice of $B$. \bigskip

Consider the following locally defined test function $U(x)$ which is a perturbation of $U_0(x)$ near the maximum point $x_0$,
\begin{align*}
\label{eq:3.5}
U(x) \coloneqq \log \left (  C_{\epsilon_0} + \tilde{\lambda}_{\text{max}}   \right ) - \frac{1}{2}  \log \left( 1 - \frac{| \nabla u |^2}{2 K}  \right )  - A u, \tag{3.5}
\end{align*}where $\tilde{\lambda}_{\text{max}}$ denotes the largest eigenvalue of $\tilde{\Lambda}$. Note that $U(x) \leq U_0(x)$ and that $G(x)$ achieves its maximum at $x_0$, where we have $U_0(x_0) = U(x_0)$. Let $\tilde{\lambda}_i$ be the eigenvalues of $\tilde{\Lambda}$. Then $\tilde{\lambda}_1 = \lambda_1$, and all the remaining eigenvalues are distinct from $\tilde{\lambda}_1$. In particular, $\tilde{\lambda}_1$ will be a smooth function near $x_0$. Now, instead of considering the elliptic operator $\Delta_{\eta}$ acting on the function $U(x)$, we perturb the elliptic operator $\Delta_{\eta}$ and get a new elliptic operator $\mathcal L_F$ which is again only defined near the maximum point $x_0$. We set

\begin{align*}
\label{eq:3.6}
\mathcal L_F \coloneqq \frac{\pa F  }{\pa \Lambda^k_i } \left ( \tilde{\Lambda}  \right )  \omega^{k \bar{j}}  \frac{\pa^2}{\pa z_i\pa\bar z_j},\tag{3.6}
\end{align*}where $F = F(\Lambda) = f(\lambda_1, \dots, \lambda_n)$ is a symmetric function defined on Hermitian matrix $\Lambda$, which only depends on the eigenvalues of $\Lambda$. We define $F$ to be
\begin{align*}
F(\Lambda) \coloneqq - \sqrt{-1} \log \left ( \frac{\det \left ( \Id + \sqrt{-1} \Lambda   \right )}{\sqrt{\det \left ( \Id +  \Lambda^2   \right )}} \right );  \ 
f(\lambda_1, \dots, \lambda_n) \coloneqq \sum_{i} \arctan (\lambda_i).
\end{align*}Note that the symmetric function $F$ is considered in Jacob--Yau \cite{jacob2017special}. We have to specify a branch cut for the logarithm by having $F = 0$ when $\Lambda = 0$.\bigskip

We can rewrite the equation as
\begin{align*}
\label{eq:3.7}
F(\Lambda) =  \Theta_{\omega}(\chi) = \sum_{i = 1}^{n} \arctan (\lambda_i) = h(x), \tag{3.7}
\end{align*}where $\Lambda = \omega^{-1} \chi$.\bigskip

The symmetric function equations were studied by Caffarelli--Nirenberg--Spruck \cite{caffarelli1985dirichlet} for the Dirichlet problem in the real case, and have been studied extensively. See, for instance, \cite{bedford1976dirichlet, guan2014dirichlet, guan2002convex, guan1994weyl, guan2003christoffel, guan2006christoffel, kolodziej1998complex, li2004dirichlet, urbas1990expansion} and the references therein. One of the most important equation will be the Monge--Ampère equation, where we take $F(\Lambda) = \log \left (  \det \Lambda   \right )$. The complex Monge--Ampère equation was first solved on compact Kähler manifolds by Yau \cite{yau1978ricci} and on compact Hermitian manifolds by Tosatti--Weinkov \cite{tosatti2010estimates, tosatti2010complex} with some earlier works by Cherrier \cite{cherrier1987equations}, Hanani \cite{hanani1996equations}, Zhang \cite{zhang2010priori} and Guan--Li \cite{guan2009complex}.\bigskip

By \hyperlink{L:2.2}{Lemma 2.2}, since the eigenvalues of $\tilde{\Lambda}$ are distinct from each other, thus the derivatives of $F$ with respect to the entries $\Lambda_i^k$ at $\tilde{\Lambda}$ almost equal to the coefficients of our original elliptic operator $\Delta_\eta$. The difference is that the eigenvalues will be the perturbed one instead.\bigskip

Now, we try to apply the maximum principle, that is, we consider the following equation later in this section,

\begin{align*}
\label{eq:3.8}
\mathcal L_F \left ( U(x) \right ) &= \mathcal L_F \left ( \log \left (  C_{\epsilon_0} + \tilde{\lambda}_{\text{max}}   \right ) - \frac{1}{2}  \log \left( 1 - \frac{| \nabla u |^2}{2 K}  \right )  - A u \right ). \tag{3.8}
\end{align*}

For the first term $\log \left (  C_{\epsilon_0} + \tilde{\lambda}_{\text{max}}   \right )$ of our test function (\ref{eq:3.8}), we can also view it as a function determined by the eigenvalues. We write
\begin{align*}
G(\Lambda) \coloneqq \log \left ( C_{\epsilon_0} + \lim_{k \rightarrow \infty} \sqrt[k]{ \left | \Tr (\Lambda^k)   \right |} \right );  \ g(\lambda_1, \dots, \lambda_n) \coloneqq \log \left ( C_{\epsilon_0} + \text{max}_i \left | \lambda_i   \right | \right ).
\end{align*}

In particular, by the phase condition and \hyperlink{L:2.3}{Lemma 2.3}, we see that
\begin{align*}
g \left (   \tilde{\lambda}_1, \dots, \tilde{\lambda}_n \right ) &= \log \left ( C_{\epsilon_0} + \tilde{\lambda}_1 \right ) = \log \left ( C_{\epsilon_0} +  \lambda_1 \right ).
\end{align*}

For the functions $f$ and $g$, by considering the variation with respect to the eigenvalues, we have the following Lemmas.

\hypertarget{L:3.2}{\begin{flemma}}
By taking $f(\lambda) = \sum_{i = 1}^n \arctan (\lambda_i)$ and $g(\lambda)=\log(C_{\epsilon_0} +\lambda_1)$, we have 
\begin{align*}
\label{eq:3.9}
f_i = \frac{1}{1+ \lambda_i^2}&, \qquad g_i = \delta_{1i} \frac{1}{C_{\epsilon_0} + \lambda_1}, \tag{3.9} \\
\label{eq:3.10}
f_{ij}=-\delta_{ij}\frac{2\lambda_i}{(1+\lambda_i^2)^2}&,\qquad g_{ij}=-\delta_{1i}\delta_{1j}\frac{1}{(C_{\epsilon_0} +\lambda_1)^2}. \tag{3.10}
\end{align*}Here $f_i \coloneqq \frac{\partial  f}{\partial \lambda_i}, g_i \coloneqq \frac{\partial g}{\partial \lambda_i}, f_{ij} \coloneqq \frac{\partial^2 f}{\partial \lambda_i \partial \lambda_j}, g_{ij} \coloneqq \frac{\partial^2 g}{\partial \lambda_i \partial \lambda_j}.$
\end{flemma}

Now, near the maximum point $x_0$, we consider the function $\tilde{h}(x) \coloneqq \sum_{i =1}^n \arctan (\tilde{\lambda}_i)$, which is defined only near $x_0$. By the following Lemma, one can see that the perturbed function $\tilde{h}(x)$ is close to our original function $h(x) = \sum_{i =1}^n \arctan (\lambda_i)$. The proof is a straightforward computation.

\hypertarget{L:3.3}{\begin{flemma}}
\begin{align*}
\label{eq:3.11}
\frac{\partial F}{\partial \Lambda^k_i} &= \left (  \left ( \Id + \Lambda^2   \right )^{-1}   \right )^i_k. \tag{3.11}
\end{align*}
\end{flemma}

Instead of the original equation $F(\Lambda) = \Theta_\omega \left ( \ubar{\chi} + \sqrt{-1} \partial \bar{\partial} u \right ) = h(x)$, we consider the perturbed one $F(\tilde{\Lambda}) = \tilde{h}(x)$ because we want to make sure that near the maximum point, our eigenvalues are differentiable. To do the $C^2$ estimate, we have to use the perturbed equation to replace the fourth derivative. We have the following Lemmas.

\hypertarget{L:3.4}{\begin{flemma}}
Let $F( \tilde{\Lambda} )= \tilde{h}(x)$, then near $x_0$, we have
\begin{align*}
\label{eq:3.12}
 \frac{\partial \tilde{h}}{\partial \bar{z}_k}  &=   \frac{\partial F}{\partial \Lambda^j_i} \left (  \tilde{\Lambda}   \right ) \frac{\partial \tilde{\Lambda}^j_i }{\partial \bar{z}_k}  = \frac{\partial F}{\partial \Lambda^j_i} \left (  \tilde{\Lambda}   \right ) \frac{\partial \Lambda^j_i }{\partial \bar{z}_k}, \tag{3.12}\\
 \label{eq:3.13}
 \frac{\partial^2 \tilde{h}}{\partial z_k \partial \bar{z}_k} &=  \frac{\partial^2 F}{\partial \Lambda^j_i \partial \Lambda_r^s} \left ( \tilde{\Lambda} \right ) \frac{\partial \tilde{\Lambda}_i^j }{\partial \bar{z}_k} \frac{\partial \tilde{\Lambda}_r^s}{\partial z_k} + \frac{\partial F}{\partial \Lambda_i^j } \left ( \tilde{\Lambda} \right ) \frac{\partial^2 \tilde{\Lambda}_i^j }{\partial z_k \partial \bar{z}_k} \tag{3.13} \\
 &=  \frac{\partial^2 F}{\partial \Lambda^j_i \partial \Lambda_r^s} \left ( \tilde{\Lambda} \right )  \frac{\partial \Lambda_i^j }{\partial \bar{z}_k} \frac{\partial \Lambda_r^s}{\partial z_k} + \frac{\partial F}{\partial \Lambda_i^j } \left ( \tilde{\Lambda} \right ) \frac{\partial^2 \Lambda_i^j }{\partial z_k \partial \bar{z}_k}. 
\end{align*}
\end{flemma}

\hypertarget{L:3.5}{\begin{flemma}}
Near the maximum point, we can rewrite the term $\frac{\partial^2 \tilde{h}}{\partial z_k \partial \bar{z}_k}$ as
\begin{align*}
\label{eq:3.14}
\frac{\partial^2 \tilde{h}}{\partial z_k \partial \bar{z}_k} =  f_{ir} \left ( \tilde{\lambda}  \right) \frac{\partial \Lambda_i^i }{ \partial \bar{z}_k} \frac{\partial \Lambda_r^r }{\partial z_k} + \sum_{i \neq j} \frac{f_i - f_j}{\tilde{\lambda}_i - \tilde{\lambda}_j} \left (  \tilde{\lambda}  \right ) \frac{\partial \Lambda_i^j}{\partial \bar{z}_k} \frac{\partial \Lambda_j^i }{\partial z_k} + f_i \left ( \tilde{\lambda} \right ) \frac{\partial^2 \Lambda_i^i }{\partial z_k \partial \bar{z}_k}. \tag{3.14}
\end{align*}
\end{flemma}

\begin{proof}
By equations (\ref{eq:2.3}) and (\ref{eq:2.4}), since $\tilde{\Lambda}$ is a diagonal matrix, we get

\begin{align*}
\frac{\partial^2 \tilde{h}}{\partial z_k \partial \bar{z}_k} = \left ( f_{ir} \left ( \tilde{\Lambda}  \right )  \delta_{ij} \delta_{rs} + \frac{f_i - f_j}{\tilde{\lambda}_i - \tilde{\lambda}_j} \left (  \tilde{\Lambda}  \right ) (1 - \delta_{ij}) \delta_{is} \delta_{jr}  \right )\frac{\partial \tilde{\Lambda}_i^j}{\partial \bar{z}_k} \frac{\partial \tilde{\Lambda}_s^r }{\partial z_k} + \delta_{ij} f_i \left (  \tilde{\Lambda}  \right ) \frac{\partial^2 \tilde{\Lambda}_j^i }{\partial z_k \partial \bar{z}_k}.
\end{align*}Simplifying it, we get
\begin{align*}
 \frac{\partial^2 \tilde{h}}{\partial z_k \partial \bar{z}_k} &=  f_{ir} \left ( \tilde{\lambda}  \right) \frac{\partial \Lambda_i^i }{ \partial \bar{z}_k} \frac{\partial \Lambda_r^r }{\partial z_k} + \sum_{i \neq j} \frac{f_i - f_j}{\tilde{\lambda}_i - \tilde{\lambda}_j} \left (  \tilde{\lambda}  \right ) \frac{\partial \Lambda_i^j}{\partial \bar{z}_k} \frac{\partial \Lambda_j^i }{\partial z_k} + f_i \left ( \tilde{\lambda} \right ) \frac{\partial^2 \Lambda_i^i }{\partial z_k \partial \bar{z}_k}. 
 \end{align*}

\end{proof}

\hypertarget{L:3.6}{\begin{flemma}}The first and second derivatives of $\Lambda$ are
\begin{align*}
\label{eq:3.15}
\frac{\partial \Lambda_i^j }{\partial \bar{z}_k} &=  \tensor[]{\omega}{^{j \bar{p}}_{, \bar{k}}} \tensor[]{\hat{\chi}}{_{i \bar{p}}} +  \tensor[]{\omega}{^{j \bar{p}}} \tensor[]{\hat{\chi}}{_{i \bar{p}, \bar{k}}} =  - \omega^{j \bar{b}} \omega_{a \bar{b}, \bar{k}} \omega^{a \bar{p}} \tensor[]{\hat{\chi}}{_{i \bar{p}}} +  \tensor[]{\omega}{^{j \bar{p}}} \tensor[]{\hat{\chi}}{_{i \bar{p}, \bar{k}}}, \tag{3.15} \\ 
\label{eq:3.16}
\frac{\partial^2 \Lambda_i^j }{\partial z_l \partial \bar{z}_k} &= \tensor[]{\omega}{^{j \bar{p}}_{, \bar{k}l}} \tensor[]{\hat{\chi}}{_{i \bar{p}}} + \tensor[]{\omega}{^{j \bar{p}}_{, \bar{k}}} \tensor[]{\hat{\chi}}{_{i \bar{p}, l}} +  \tensor[]{\omega}{^{j \bar{p}}_{, l}} \tensor[]{\hat{\chi}}{_{i \bar{p}, \bar{k}}} + \tensor[]{\omega}{^{j \bar{p}}} \tensor[]{\hat{\chi}}{_{i \bar{p}, \bar{k}l}} \tag{3.16}  \\
&=  \omega^{j \bar{d}} \omega_{c \bar{d}, l} \omega^{c \bar{b}} \omega_{a \bar{b}, \bar{k}} \omega^{a \bar{p}} \hat{\chi}_{i \bar{p}} - \omega^{j \bar{b}} \omega_{a \bar{b}, \bar{k} l} \omega^{a \bar{p}} \hat{\chi}_{i \bar{p}} + \omega^{j \bar{b}} \omega_{a \bar{b}, \bar{k}} \omega^{a \bar{d}} \omega_{c \bar{d}, l} \omega^{c \bar{p}} \hat{\chi}_{i \bar{p}} \\
&\quad - \omega^{j \bar{b}} \omega_{a \bar{b}, \bar{k}} \omega^{a \bar{p}} \hat{\chi}_{i \bar{p}, l} - \omega^{j \bar{b}} \omega_{a \bar{b}, l} \omega^{a \bar{p}} \hat{\chi}_{i \bar{p}, \bar{k}} + \omega^{j \bar{p}} \hat{\chi}_{i \bar{p}, \bar{k} l}, 
\end{align*}where we denote $\hat{\chi}_{i \bar{j}} = \ubar{\chi}_{i \bar{j}} + u_{i \bar{j}}$.
\end{flemma}

By \hyperlink{L:3.3}{Lemma 3.3} and \hyperlink{L:3.8}{Lemma 3.8}, at the maximum point,  direct computations give us the following lemma.

\hypertarget{L:3.7}{\begin{flemma}}At the maximum point $x_0$, we have
\begin{align*}
\label{eq:3.17}
\frac{\partial \Lambda_i^j}{\partial \bar{z}_k} (x_0) &= - \lambda_i \omega_{i \bar{j}, \bar{k}}  +  \tensor[]{\hat{\chi}}{_{i \bar{j}, \bar{k}}} = - \tilde{\lambda}_i \omega_{i \bar{j}, \bar{k}} - B_{ii} \omega_{i \bar{j}, \bar{k}} +   \tensor[]{\hat{\chi}}{_{i \bar{j}, \bar{k}}}, \tag{3.17} \\ 
\label{eq:3.18}
\frac{\partial^2 \Lambda_i^j}{\partial z_l \partial \bar{z}_k} (x_0) &= \lambda_i \omega_{i \bar{a}, \bar{k}} \omega_{a \bar{j}, l} + \lambda_i \omega_{a \bar{j}, \bar{k}} \omega_{i \bar{a}, l} - \lambda_i \omega_{i \bar{j}, \bar{k} l}  - \omega_{a \bar{j}, \bar{k}}  \tensor[]{\hat{\chi}}{_{i \bar{a}, l}} - \omega_{a \bar{j}, l}  \tensor[]{\hat{\chi}}{_{i \bar{a}, \bar{k}}}  +  \tensor[]{\hat{\chi}}{_{i \bar{j}, \bar{k}l}} \tag{3.18} \\ 
&= \tilde{\lambda}_i \omega_{i \bar{a}, \bar{k}} \omega_{a \bar{j}, l} + \tilde{\lambda}_i \omega_{a \bar{j}, \bar{k}} \omega_{i \bar{a}, l} - \tilde{\lambda}_i \omega_{i \bar{j}, \bar{k} l}  + B_{ii} \omega_{i \bar{a}, \bar{k}} \omega_{a \bar{j}, l} + B_{ii} \omega_{a \bar{j}, \bar{k}} \omega_{i \bar{a}, l}  \\ 
&\quad \quad \quad \quad  \quad \quad \quad \quad \quad \quad \quad \quad    - B_{ii} \omega_{i \bar{j}, \bar{k} l}  - \omega_{a \bar{j}, \bar{k}}  \tensor[]{\hat{\chi}}{_{i \bar{a}, l}} - \omega_{a \bar{j}, l}  \tensor[]{\hat{\chi}}{_{i \bar{a}, \bar{k}}}  +  \tensor[]{\hat{\chi}}{_{i \bar{j}, \bar{k}l}}.   
\end{align*}
\end{flemma}

\hypertarget{L:3.8}{\begin{flemma}}
By Lemmas \hyperlink{L:2.3}{2.3}, \hyperlink{L:3.1}{3.1}, \hyperlink{L:3.2}{3.2}, \hyperlink{L:3.5}{3.5}, and \hyperlink{L:3.6}{3.6}, at the maximum point, we get
\begin{align*}
\label{eq:3.19}
\frac{\partial^2 \tilde{h}}{\partial z_k \partial \bar{z}_k} 
&= - \sum_{i+j < 2n} \frac{\tilde{\lambda}_i + \tilde{\lambda}_j }{(1 + \tilde{\lambda}_i^2)(1 + \tilde{\lambda}_j^2)} \left | \left( \frac{1 - \tilde{\lambda}_i \tilde{\lambda}_j }{ \tilde{\lambda}_i + \tilde{\lambda}_j} - B_{jj} \right ) \omega_{j \bar{i}, k}   +  \tensor[]{\hat{\chi}}{_{j \bar{i},  {k}}}   \right |^2  - \frac{2 \tilde{\lambda}_n \left |  \tensor[]{\hat{\chi}}{_{n \bar{n}, \bar{k}}}   \right |^2 }{(1 + \tilde{\lambda}_n^2)^2} \tag{3.19} \\  
&\quad+ \frac{1}{1+ \tilde{\lambda}_i^2} \tensor[]{\hat{\chi}}{_{i \bar{i},   \bar{k} k}}    - \frac{\tilde{\lambda}_i}{1+ \tilde{\lambda}_i^2} \omega_{i \bar{i}, \bar{k} k}     - B_{ii} \frac{1}{1 + \tilde{\lambda}_i^2} \omega_{i \bar{i}, \bar{k} k}   +   \sum_{i+j < 2n} \frac{ 1 }{ \lambda_i + \lambda_j } \left | \omega_{j \bar{i}, k}  \right |^2. 
\end{align*}
\end{flemma}

\begin{proof}By \hyperlink{L:3.5}{Lemma 3.5} and \hyperlink{L:3.6}{Lemma 3.6}, we have
\begin{align*}
\frac{\partial^2 \tilde{h}}{\partial z_k \partial \bar{z}_k} 
&= - \frac{2 \tilde{\lambda}_i}{(1 + \tilde{\lambda}_i^2)^2} \left ( - \lambda_i \omega_{i \bar{i}, \bar{k}}  + \tensor[]{\hat{\chi}}{_{i \bar{i}, \bar{k}}}   \right ) \left ( - \lambda_i \omega_{i \bar{i},  {k}}  + \tensor[]{\hat{\chi}}{_{i \bar{i},  {k}}}   \right ) \\
&\quad- \sum_{i \neq j } \frac{\tilde{\lambda}_i + \tilde{\lambda}_j}{(1 + \tilde{\lambda}_i^2)(1 + \tilde{\lambda}_j^2)} \left (  - \lambda_i \omega_{i \bar{j}, \bar{k}}  + \tensor[]{\hat{\chi}}{_{i \bar{j}, \bar{k}}}  \right ) \left ( - \lambda_j  \omega_{j \bar{i}, k}  + \tensor[]{\hat{\chi}}{_{j \bar{i},  {k}}}   \right ) \\
&\quad+ \frac{1}{1+ \tilde{\lambda}_i^2} \left (  \lambda_i  \left (  \left | \omega_{j \bar{i}, k} \right |^2 + \left | \omega_{i \bar{j}, k}    \right |^2  \right ) - \lambda_i  \omega_{i \bar{i}, \bar{k} k} -  \omega_{j \bar{i}, \bar{k}} \tensor[]{\hat{\chi}}{_{i \bar{j}, k}} - \omega_{j \bar{i},  {k}} \tensor[]{\hat{\chi}}{_{i \bar{j}, \bar{k}}} +   \tensor[]{\hat{\chi}}{_{i \bar{i}, \bar{k}k}}   \right ) \\
&= -   \frac{ \tilde{\lambda}_i + \tilde{\lambda}_j }{(1 + \tilde{\lambda}_i^2)(1 + \tilde{\lambda}_j^2)} \left ( \omega_{j \bar{i}, k} - B_{jj} \omega_{j \bar{i}, k} +  \tensor[]{\hat{\chi}}{_{j \bar{i},  {k}}}   \right ) \left ( \omega_{i \bar{j}, \bar{k}} - B_{ii} \omega_{i \bar{j}, \bar{k}} + \tensor[]{\hat{\chi}}{_{i \bar{j}, \bar{k}}}   \right ) \\ 
&\quad+  2   \frac{ \tilde{\lambda}_i + \tilde{\lambda}_j }{(1 + \tilde{\lambda}_i^2)(1 + \tilde{\lambda}_j^2)} \left | \omega_{j \bar{i}, k}  \right |^2 - \frac{ \tilde{\lambda}_i}{1 + \tilde{\lambda}_i^2} \omega_{i \bar{i}, \bar{k} k} - B_{ii} \frac{1}{1 + \tilde{\lambda}_i^2} \omega_{i \bar{i}, \bar{k} k}+ \frac{1}{1+ \tilde{\lambda}_i^2} \tensor[]{\hat{\chi}}{_{i \bar{i},   \bar{k} k}}  \\
&\quad+   \frac{\tilde{\lambda}_i \tilde{\lambda}_j + \tilde{\lambda}_i + \tilde{\lambda}_j - 1}{(1+ \tilde{\lambda}_i^2)(1 + \tilde{\lambda}_j^2)} \left (  \omega_{j \bar{i}, k} \tensor[]{\hat{\chi}}{_{i \bar{j},  \bar{k}}}  +  \omega_{i \bar{j}, \bar{k}} \tensor[]{\hat{\chi}}{_{j \bar{i},   {k}}}  - (B_{ii} + B_{jj}) \left | \omega_{j \bar{i}, k} \right |^2 \right ).
\end{align*}

Note that by our choice of \hyperlink{tilde Lambda}{$\tilde{\Lambda}$} and \hyperlink{L:2.3}{Lemma 2.3}, no matter what the sign of $\tilde{\lambda}_n$ is, we always have $\tilde{\lambda}_{n-1} + \tilde{\lambda}_n > \tan(\frac{\epsilon_0}{4}) > 0$, which implies the following factorization

\begin{align*}
\quad \frac{\partial^2 \tilde{h}}{\partial z_k \partial \bar{z}_k}  
 &= - \sum_{i+j < 2n} \frac{\tilde{\lambda}_i + \tilde{\lambda}_j }{(1 + \tilde{\lambda}_i^2)(1 + \tilde{\lambda}_j^2)} \left | \frac{1 - \tilde{\lambda}_i \tilde{\lambda}_j }{ \tilde{\lambda}_i + \tilde{\lambda}_j} \omega_{j \bar{i}, k} - B_{jj} \omega_{j \bar{i}, k} +  \tensor[]{\hat{\chi}}{_{j \bar{i},  {k}}}   \right |^2 - \frac{2 \tilde{\lambda}_n}{(1 + \tilde{\lambda}_n^2)^2} \left |  \tensor[]{\hat{\chi}}{_{n \bar{n}, \bar{k}}}   \right |^2 \\ 
&\quad     + \frac{1}{1+ \tilde{\lambda}_i^2} \tensor[]{\hat{\chi}}{_{i \bar{i},   \bar{k} k}}     - \frac{\tilde{\lambda}_i}{1+ \tilde{\lambda}_i^2} \omega_{i \bar{i}, \bar{k} k} - B_{ii} \frac{1}{1 + \tilde{\lambda}_i^2} \omega_{i \bar{i}, \bar{k} k}   +   \sum_{i+j < 2n} \frac{ 1 }{ \lambda_i + \lambda_j } \left | \omega_{j \bar{i}, k}  \right |^2.
\end{align*}

\end{proof}

Now, with all these tools inside our toolbox, we can prove the following $C^2$ estimate when $\omega$ is merely a Hermitian metric, we have the following

\hypertarget{T:3.2}{\begin{ftheorem}}
Let $X^n$ be a connected compact Hermitian manifold with Hermitian form $\omega$. Suppose $u \colon X \rightarrow \mathbb{R}$ is a smooth function solving $\Theta_{\omega}(\ubar{\chi} + \sqrt{-1} \partial \bar{\partial} u) = h(x)$, where $h \colon X \rightarrow [(n-2)\frac{\pi}{2} + \epsilon_0, n \frac{\pi}{2} )$ for some $\epsilon_0 > 0$ and $\ubar{\chi}$ is a $C$-subsolution. Then there exists a constant $C$ such that 
\begin{align*}
\label{eq:3.20}
|\partial \bar{\partial} u |_{\omega} \leq C \left ( 1 + \sup_X |\nabla u|^2_{\omega} \right), \tag{3.20}
\end{align*}where $C = C\left ( |h|_{C^2(X, \omega)}, \osc_X u, \ubar{\chi}, \omega, \epsilon_0  \right)$ and $\nabla$ is the Chern connection with respect to $\omega$.\bigskip
\end{ftheorem} 

\begin{proof}
As we mentioned, we consider the test function,
\begin{align*}
U(x) \coloneqq \log \left (  C_{\epsilon_0} + \tilde{\lambda}_{\text{max}}   \right ) - \frac{1}{2}  \log \left( 1 - \frac{| \nabla u |^2}{2 K}  \right )  - A u, 
\end{align*}where we define $K \coloneqq 1 + \sup_X \left | \nabla u \right |^2_{\omega}$ and $A$ is a constant to be determined later. For convenience, let us normalize $u$ so that $\inf_X u =0$. \bigskip

At the maximum point $x_0$ of our test function, we have

\begin{align*}
\mathcal{L}_F \left (  G( \tilde{\Lambda} )   \right )   
&=f_i (\tilde{\lambda}) g_{rs} (\tilde{\lambda}) \frac{\partial \Lambda_r^r}{\partial z_i} \frac{\partial \Lambda_s^s}{\partial \bar{z}_i} + f_i (\tilde{\lambda}) \sum_{r \neq s} \frac{g_r - g_s}{\tilde{\lambda}_r - \tilde{\lambda}_s} (\tilde{\lambda}) \frac{\partial \Lambda_s^r}{\partial z_i} \frac{\partial \Lambda_r^s}{\partial \bar{z}_i} + f_i (\tilde{\lambda}) g_r (\tilde{\lambda}) \frac{\partial^2 \Lambda_r^r}{\partial z_i \partial \bar{z}_i}    \\
&=  -\frac{1}{(1+ \tilde{\lambda}_i^2)(C_{\epsilon_0} + \tilde{\lambda}_1)^2} \left | \omega_{1 \bar{1}, i} - B_{11} \omega_{1 \bar{1}, i} + \tensor[]{\hat{\chi}}{_{1 \bar{1},  {i}}} \right |^2 + \frac{ 1 - \tilde{\lambda}_1^2 }{(1 + \tilde{\lambda}_i^2)(C_{\epsilon_0} + \tilde{\lambda}_1)^2} \left |  \omega_{1 \bar{1}, i}  \right |^2 \\
&\quad + \sum_{s \neq 1} \frac{1}{(1 + \tilde{\lambda}_i^2)(C_{\epsilon_0} + \tilde{\lambda}_1)( \tilde{\lambda}_1 - \tilde{\lambda}_s)} \left | -\tilde{\lambda}_1 \omega_{s \bar{1}, i} - B_{ss} \omega_{s \bar{1}, i} +   \tensor[]{\hat{\chi}}{_{s \bar{1},  {i}}} \right |^2 \\
&\quad + \sum_{s \neq 1} \frac{1}{(1 + \tilde{\lambda}_i^2)(C_{\epsilon_0} + \tilde{\lambda}_1)( \tilde{\lambda}_1 - \tilde{\lambda}_s)} \left | -\tilde{\lambda}_1 \omega_{1 \bar{s}, i} - B_{11} \omega_{1 \bar{s}, i} +   \tensor[]{\hat{\chi}}{_{1 \bar{s},  {i}}} \right |^2 \\
&\quad- \sum_{s \neq 1} \frac{\tilde{\lambda}_1}{(1 + \tilde{\lambda}_i^2)(C_{\epsilon_0} + \tilde{\lambda}_1)} \left |  \omega_{s \bar{1}, i}   \right |^2 - \sum_{s \neq 1} \frac{\tilde{\lambda}_1}{(1 + \tilde{\lambda}_i^2)(C_{\epsilon_0} + \tilde{\lambda}_1)} \left |  \omega_{1 \bar{s}, i}   \right |^2 \\
&\quad+ \frac{1}{(1 + \tilde{\lambda}_i^2)(C_{\epsilon_0} + \tilde{\lambda}_1)} \left ( \tilde{\lambda}_1 \left (  \left |  \omega_{j \bar{1}, i}  \right |^2  + \left | \omega_{1 \bar{j}, i}  \right |^2  \right ) - \tilde{\lambda}_1 \omega_{1 \bar{1}, \bar{i} i} - B_{11} \omega_{1 \bar{1}, \bar{i} i}   +   \tensor[]{\hat{\chi}}{_{1 \bar{1},  \bar{i} i}}   \right ).
\end{align*}By simplifying it and by our choice of \hyperlink{L:3.1}{coordinate}, at the maximum point, we get

\begin{align*}
\label{eq:3.21}
\mathcal{L}_F \left (  G( \tilde{\Lambda} )   \right ) 
&\geq -C_1  -\frac{1}{(1+ \tilde{\lambda}_i^2)(C_{\epsilon_0} + \tilde{\lambda}_1)^2} \left |   \tensor[]{\hat{\chi}}{_{1 \bar{1},  {i}}} \right |^2  + \frac{1}{(1 + \tilde{\lambda}_i^2)(C_{\epsilon_0} + \tilde{\lambda}_1)}  u_{1 \bar{1} \bar{i} i}  \tag{3.21} \\
&\quad+ \sum_{s \neq 1} \frac{1}{(1 + \tilde{\lambda}_i^2)(C_{\epsilon_0} + \tilde{\lambda}_1)(\tilde{\lambda}_1 - \tilde{\lambda}_s)}  \left | -\tilde{\lambda}_1 \omega_{s \bar{1}, i} - B_{ss} \omega_{s \bar{1}, i} +   \tensor[]{\hat{\chi}}{_{s \bar{1},  {i}}} \right |^2  \\
&\quad+ \sum_{s \neq 1} \frac{1}{(1 + \tilde{\lambda}_i^2)(C_{\epsilon_0} + \tilde{\lambda}_1)(\tilde{\lambda}_1 - \tilde{\lambda}_s)}  \left | -\tilde{\lambda}_1 \omega_{1 \bar{s}, i}     +  \tensor[]{\hat{\chi}}{_{1 \bar{s},  {i}}} \right |^2.    
\end{align*}

Combining (\ref{eq:3.19}) and (\ref{eq:3.21}), we get,

\begingroup
\allowdisplaybreaks
\begin{align*}
\label{eq:3.22}
\mathcal{L}_F \left (  G( \tilde{\Lambda} )   \right )  
&\geq -C_2  -\frac{1}{(1 + \tilde{\lambda}_i^2)(C_{\epsilon_0} + \tilde{\lambda}_1)^2} \left |   \tensor[]{\hat{\chi}}{_{1 \bar{1},  {i}}} \right |^2  \tag{3.22}   \\
&\quad+ \sum_{s \neq 1} \frac{1}{(1 + \tilde{\lambda}_i^2)(C_{\epsilon_0} + \tilde{\lambda}_1)(\tilde{\lambda}_1 - \tilde{\lambda}_s)}   \left | -\tilde{\lambda}_1 \omega_{s \bar{1}, i} - B_{ss} \omega_{s \bar{1}, i} +   \tensor[]{\hat{\chi}}{_{s \bar{1},  {i}}} \right |^2   \\
&\quad+ \sum_{s \neq 1} \frac{1}{(1 + \tilde{\lambda}_i^2)(C_{\epsilon_0} + \tilde{\lambda}_1)(\tilde{\lambda}_1 - \tilde{\lambda}_s)}   \left | -\tilde{\lambda}_1 \omega_{1 \bar{s}, i}    +  \tensor[]{\hat{\chi}}{_{1 \bar{s},  {i}}} \right |^2     \\
&\quad+ \sum_{i+j < 2n} \frac{\tilde{\lambda}_i + \tilde{\lambda}_j }{(C_{\epsilon_0} + \tilde{\lambda}_1)(1 + \tilde{\lambda}_i^2)(1 + \tilde{\lambda}_j^2)} \left | \frac{1 - \tilde{\lambda}_i \tilde{\lambda}_j }{ \tilde{\lambda}_i + \tilde{\lambda}_j} \omega_{j \bar{i}, 1} - B_{jj} \omega_{j \bar{i}, 1} +  \tensor[]{\hat{\chi}}{_{j \bar{i},  {1}}}   \right |^2 \\
&\quad- \sum_{i +j <2n} \frac{1}{(C_{\epsilon_0} + \tilde{\lambda}_1)( \tilde{\lambda}_i + \tilde{\lambda}_j)} \left | \omega_{j \bar{i}, 1}    \right |^2 + \frac{2 \tilde{\lambda}_n}{(C_{\epsilon_0} + \tilde{\lambda}_1)(1 + \tilde{\lambda}_n^2)^2} \left |  \tensor[]{\hat{\chi}}{_{n \bar{n}, \bar{1}}}   \right |^2.
\end{align*}
\endgroup

Note that the phase condition ${\epsilon_0} > 0$ is crucial here, otherwise the terms $\frac{1}{(C_{\epsilon_0} + \tilde{\lambda}_1)( \tilde{\lambda}_i + \tilde{\lambda}_j)} \left | \omega_{j \bar{i}, 1}    \right |^2$ might blow up. With \hyperlink{L:2.3}{Lemma 2.3}, we know $\tilde{\lambda}_i + \tilde{\lambda}_j > \tan \left(  \frac{{\epsilon_0}}{4}  \right ) > 0$ if $i + j < 2n$. Thus, we have

\begingroup
\allowdisplaybreaks
\begin{align*}
\label{eq:3.23}
&\mathcal{L}_F \left (  G( \tilde{\Lambda} )   \right ) \tag{3.23} \\
&\geq -C_3  -\frac{1}{(1 + \tilde{\lambda}_i^2)(C_{\epsilon_0} + \tilde{\lambda}_1)^2} \left |   \tensor[]{\hat{\chi}}{_{1 \bar{1},  {i}}} \right |^2  + \frac{2 \tilde{\lambda}_n}{(C_{\epsilon_0} + \tilde{\lambda}_1)(1 + \tilde{\lambda}_n^2)^2} \left |  \tensor[]{\hat{\chi}}{_{n \bar{n}, \bar{1}}}   \right |^2   \\
&\quad+ \sum_{s \neq 1} \frac{1}{(1 + \tilde{\lambda}_i^2)(C_{\epsilon_0} + \tilde{\lambda}_1)(\tilde{\lambda}_1 - \tilde{\lambda}_s)}  \left | -\tilde{\lambda}_1 \omega_{s \bar{1}, i} - B_{ss} \omega_{s \bar{1}, i} +   \tensor[]{\hat{\chi}}{_{s \bar{1},  {i}}} \right |^2   \\
&\quad+ \sum_{s \neq 1} \frac{1}{(1 + \tilde{\lambda}_i^2)(C_{\epsilon_0} + \tilde{\lambda}_1)(\tilde{\lambda}_1 - \tilde{\lambda}_s)}   \left | -\tilde{\lambda}_1 \omega_{1 \bar{s}, i}     +  \tensor[]{\hat{\chi}}{_{1 \bar{s},  {i}}} \right |^2     \\
&\quad+ \sum_{i+j < 2n} \frac{\tilde{\lambda}_i + \tilde{\lambda}_j }{(C_{\epsilon_0} + \tilde{\lambda}_1)(1 + \tilde{\lambda}_i^2)(1 + \tilde{\lambda}_j^2)} \left | \frac{1 - \tilde{\lambda}_i \tilde{\lambda}_j }{ \tilde{\lambda}_i + \tilde{\lambda}_j} \omega_{j \bar{i}, 1} - B_{jj} \omega_{j \bar{i}, 1} +  \tensor[]{\hat{\chi}}{_{j \bar{i},  {1}}}   \right |^2 \\
&=  -C_3  -\frac{1}{(1+ \tilde{\lambda}_1^2)(C_{\epsilon_0} + \tilde{\lambda}_1)^2} \left |   \tensor[]{\hat{\chi}}{_{1 \bar{1},  {1}}} \right |^2  +     \frac{ 2 \tilde{\lambda}_i }{(C_{\epsilon_0} + \tilde{\lambda}_1)(1 + \tilde{\lambda}_i^2)^2} \left |   \tensor[]{\hat{\chi}}{_{i \bar{i},  {1}}}   \right |^2   \\
&\quad- \sum_{s \neq 1} \frac{1}{(1 + \tilde{\lambda}_s^2)(C_{\epsilon_0} + \tilde{\lambda}_1)^2} \left |   \tensor[]{\hat{\chi}}{_{1 \bar{1},  {s}}} \right |^2 \\ 
&\quad+ \sum_{s \neq 1} \frac{1 }{(C_{\epsilon_0} + \tilde{\lambda}_1)(1 + \tilde{\lambda}_s^2)( \tilde{\lambda}_1 - \tilde{\lambda}_s)}   \left | - B_{ss} \omega_{s \bar{1}, 1} +   \tensor[]{\hat{\chi}}{_{s \bar{1},  {1}}}  \right |^2 \\  
&\quad-\sum_{s \neq 1} \frac{ \tilde{\lambda}_s }{(C_{\epsilon_0} + \tilde{\lambda}_1)(1 + \tilde{\lambda}_s^2)( \tilde{\lambda}_1 - \tilde{\lambda}_s)}     \omega_{s \bar{1}, {1}} \left ( -B_{ss} \omega_{1 \bar{s}, \bar{1}}  + \tensor[]{\hat{\chi}}{_{1 \bar{s},  \bar{1}}} \right )      \\
&\quad-\sum_{s \neq 1} \frac{ \tilde{\lambda}_s }{(C_{\epsilon_0} + \tilde{\lambda}_1)(1 + \tilde{\lambda}_s^2)( \tilde{\lambda}_1 - \tilde{\lambda}_s)}      \omega_{1 \bar{s}, \bar{1}} \left ( -B_{ss} \omega_{s \bar{1}, 1}  +  \tensor[]{\hat{\chi}}{_{s \bar{1},  {1}}}  \right )   \\
&\quad+ \sum_{s \neq 1} \frac{ \tilde{\lambda}_1 - \tilde{\lambda}_s + 2 \tilde{\lambda}_1 \tilde{\lambda}_s^2   }{(C_{\epsilon_0} + \tilde{\lambda}_1)(1 + \tilde{\lambda}_s^2 )( \tilde{\lambda}_1 - \tilde{\lambda}_s)( \tilde{\lambda}_1 + \tilde{\lambda}_s)} \left |   \omega_{s \bar{1}, 1}   \right |^2   \\
&=  -C_3  -\frac{1}{(1+ \tilde{\lambda}_1^2)(C_{\epsilon_0} + \tilde{\lambda}_1)^2} \left |   \tensor[]{\hat{\chi}}{_{1 \bar{1},  {1}}} \right |^2 +   \frac{ 2 \tilde{\lambda}_i }{(C_{\epsilon_0} + \tilde{\lambda}_1)(1 + \tilde{\lambda}_i^2)^2} \left |   \tensor[]{\hat{\chi}}{_{i \bar{i},  {1}}}   \right |^2    \\
&\quad+ \sum_{s \neq 1} \frac{C_{\epsilon_0} + \tilde{\lambda}_s}{(C_{\epsilon_0} + \tilde{\lambda}_1)^2(1 + \tilde{\lambda}_s^2)( \tilde{\lambda}_1 - \tilde{\lambda}_s)}   \left |   \tensor[]{\hat{\chi}}{_{1 \bar{1},  {s}}} + \frac{C_{\epsilon_0} + \tilde{\lambda}_1}{C_{\epsilon_0} + \tilde{\lambda}_s} \left (   S_{s1} - B_{ss} \omega_{s \bar{1}, 1}   - \tilde{\lambda}_s \omega_{s \bar{1}, 1}     \right ) \right |^2     \\
&\quad- \sum_{s \neq 1} \frac{ 1 }{(C_{\epsilon_0} + \tilde{\lambda}_1)(C_{\epsilon_0} + \tilde{\lambda}_s)(1 + \tilde{\lambda}_s^2)} \left |  S_{s1} - B_{ss} \omega_{s \bar{1}, 1}  - \tilde{\lambda}_s \omega_{s \bar{1}, 1}   \right |^2 \\ 
&\quad+ \frac{1}{(C_{\epsilon_0} + \tilde{\lambda}_1)(\tilde{\lambda}_1 + \tilde{\lambda}_s)} \left | \omega_{s \bar{1}, 1} \right |^2, 
\end{align*}
\endgroup
where we set $S_{s1} = \tensor[]{\hat{\chi}}{_{s \bar{1},  {1}}} - \tensor[]{\hat{\chi}}{_{1 \bar{1},  {s}}}$. In conclusion, by dropping off some positive terms in (\ref{eq:3.23}) and bounding some terms in (\ref{eq:3.23}), we get a relatively simplified inequality comparing to \cite{collins20151}.

\begin{align*}
\label{eq:3.24}
\mathcal{L}_F \left (  G( \tilde{\Lambda} )   \right ) \geq   -C_4  -\frac{1}{(1+ \tilde{\lambda}_1^2)(C_{\epsilon_0} + \tilde{\lambda}_1)^2} \left |   \tensor[]{\hat{\chi}}{_{1 \bar{1},  {1}}} \right |^2 +    \frac{ 2 \tilde{\lambda}_i }{(C_{\epsilon_0} + \tilde{\lambda}_1)(1 + \tilde{\lambda}_i^2)^2} \left |   \tensor[]{\hat{\chi}}{_{i \bar{i},  {1}}}   \right |^2. \tag{3.24}
\end{align*}There are two cases to be considered, $\tilde{\lambda}_n > 0$ or $\tilde{\lambda}_n \leq 0$.\bigskip

$\bullet$ For the case $\tilde{\lambda}_n > 0$, by dropping off some positive terms in (\ref{eq:3.24}), we have
\begin{align*}
\label{eq:3.25}
\mathcal{L}_F \left (  G( \tilde{\Lambda} )   \right )  
&\geq -C_4  -\frac{1}{(1+ \tilde{\lambda}_1^2)(C_{\epsilon_0} + \tilde{\lambda}_1)^2} \left |   \tensor[]{\hat{\chi}}{_{1 \bar{1},  {1}}} \right |^2 +    \frac{ 2 \tilde{\lambda}_1 }{(C_{\epsilon_0} + \tilde{\lambda}_1)(1 + \tilde{\lambda}_1^2)^2} \left |   \tensor[]{\hat{\chi}}{_{1 \bar{1},  {1}}}   \right |^2 \tag{3.25} \\
&= - C_4 + \frac{ \tilde{\lambda}_1^2 + 2 C_{\epsilon_0} \tilde{\lambda}_1 -1}{(C_{\epsilon_0} + \tilde{\lambda}_1)^2 (1+ \tilde{\lambda}_1^2)^2}  \left |  \tensor[]{\hat{\chi}}{_{1 \bar{1}, 1}}  \right |^2 \geq -C_4,    
\end{align*}provided that $\tilde{\lambda}_1 = \lambda_1$ is sufficiently large, otherwise we get the estimate of $\left | \partial \bar{\partial} u \right |_{\omega}$ immediately.

In conclusion, we see that for the case $\tilde{\lambda}_n > 0$, provided that $\lambda_n$ is sufficiently large, we get
\begin{align*}
\label{eq:3.26}
\mathcal{L}_F \left (  G( \tilde{\Lambda} )   \right )  
&\geq -C_4.    \tag{3.26} \\
\end{align*}

$\bullet$ For the case $\tilde{\lambda}_n \leq 0$, the estimate is much more harder. First we have
\begin{align*}
\label{eq:3.27}
&\frac{ 2 \tilde{\lambda}_i}{(C_{\epsilon_0} + \tilde{\lambda}_1)(1 + \tilde{\lambda}_i^2)^2} \left|  \tensor[]{\hat{\chi}}{_{i \bar{i}, 1}} \right|^2 \tag{3.27} \\
&= \sum_{q < n} \frac{1}{C_{\epsilon_0} + \tilde{\lambda}_1}\frac{ 2 \tilde{\lambda}_q}{(1 + \tilde{\lambda}_q^2)^2}\left|  \tensor[]{\hat{\chi}}{_{q \bar{q}, 1}}  \right|^2 + \frac{1}{C_{\epsilon_0} + \tilde{\lambda}_1} \frac{ 2 \tilde{\lambda}_n}{(1 + \tilde{\lambda}_n^2)^2} \left |   \tensor[]{\hat{\chi}}{_{n \bar{n}, 1}} \right |^2    \\
&= \sum_{q < n} \frac{1}{C_{\epsilon_0} + \tilde{\lambda}_1}\frac{2 \tilde{\lambda}_q}{(1 + \tilde{\lambda}_q^2)^2}\left|  \tensor[]{\hat{\chi}}{_{q \bar{q}, 1}}  \right|^2 + \frac{2 \tilde{\lambda}_n}{C_{\epsilon_0} + \tilde{\lambda}_1}   \left | \frac{\partial \tilde{h}}{\partial z_1} - \sum_{q < n}\frac{1}{1+ \tilde{\lambda}_q^2} \tensor[]{\hat{\chi}}{_{q \bar{q}, 1}} \right |^2.
\end{align*}The last equality is obtained by equation (\ref{eq:3.12}). Moreover, by Cauchy--Schwartz inequality, we have

\begin{align*}
\label{eq:3.28}
&\sum_{q < n} \frac{1}{C_{\epsilon_0} + \tilde{\lambda}_1}\frac{2 \tilde{\lambda}_q}{(1 + \tilde{\lambda}_q^2)^2}\left|  \tensor[]{\hat{\chi}}{_{q \bar{q}, 1}}  \right|^2 + \frac{2 \tilde{\lambda}_n}{C_{\epsilon_0} + \tilde{\lambda}_1}   \left | \frac{\partial \tilde{h}}{\partial z_1} - \sum_{q < n}\frac{1}{1+ \tilde{\lambda}_q^2} \tensor[]{\hat{\chi}}{_{q \bar{q}, 1}} \right |^2  \tag{3.28} \\
&\geq    \frac{2}{C_{\epsilon_0} + \tilde{\lambda}_1} \left ( \frac{1}{\tan ( \frac{{\epsilon_0}}{2} )}  \left | \frac{\partial \tilde{h}}{\partial z_1} \right |^2  + \sum_{q < n}\frac{ \tilde{\lambda}_q}{(1 + \tilde{\lambda}_q^2)^2} \left | \tensor[]{\hat{\chi}}{_{q \bar{q}, 1}}  \right |^2 \right ) \left ( \tan ( \frac{{\epsilon_0}}{2} ) \tilde{\lambda}_n + \sum_{q < n} \frac{\tilde{\lambda}_n}{\tilde{\lambda}_q}   \right ) \\ 
&\quad+ \sum_{q < n} \frac{1}{C_{\epsilon_0} + \tilde{\lambda}_1}\frac{2 \tilde{\lambda}_q}{(1 + \tilde{\lambda}_q^2)^2}\left| \tensor[]{\hat{\chi}}{_{q \bar{q}, 1}} \right|^2.
\end{align*}

Note that due to \hyperlink{L:2.3}{Lemma 2.3} and the choice of our perturbation \hyperlink{tilde Lambda}{$\tilde{\Lambda}$}, we have 
\begin{align*}
\sum_i \frac{1}{\tilde{\lambda}_i} = \frac{\sigma_{n-1} (\tilde{\lambda})}{ \tilde{\lambda}_1 \cdot \tilde{\lambda}_2 \cdot \dots \cdot \tilde{\lambda}_n} = \frac{\sigma_{n-1}(\tilde{\lambda})}{\sigma_n (\tilde{\lambda})} \leq -\tan ( \frac{{\epsilon_0}}{2} ) < 0,
\end{align*}thus
\begin{align*}
\label{eq:3.29}
\sum_{q < n} \frac{1}{\tilde{\lambda}_q} \leq - \frac{1}{\tilde{\lambda}_n} - \tan ( \frac{{\epsilon_0}}{2} ) \Rightarrow \sum_{q < n} \frac{\tilde{\lambda}_n}{\tilde{\lambda}_q} \geq - 1 - \tilde{\lambda}_n \tan ( \frac{{\epsilon_0}}{2} ). \tag{3.29}
\end{align*}In conclusion, with (\ref{eq:3.27}), (\ref{eq:3.28}) and (\ref{eq:3.29}), we get

\begin{align*}
\label{eq:3.30}
\frac{ 2 \tilde{\lambda}_i}{(C_{\epsilon_0} + \tilde{\lambda}_1)(1 + \tilde{\lambda}_i^2)^2} \left|  \tensor[]{\hat{\chi}}{_{i \bar{i}, 1}} \right|^2 &\geq \frac{2}{C_{\epsilon_0} + \tilde{\lambda}_1}  \frac{1}{\tan ( \frac{{\epsilon_0}}{2} )}  \left | \frac{\partial \tilde{h}}{\partial z_1} \right |^2 \left ( \tilde{\lambda}_n \tan ( \frac{{\epsilon_0}}{2} ) + \sum_{q < n} \frac{ \tilde{\lambda}_n}{ \tilde{\lambda}_q}   \right ) \tag{3.30} \\ 
\geq  &- \frac{2}{C_{\epsilon_0} + \tilde{\lambda}_1}  \frac{1}{\tan ( \frac{{\epsilon_0}}{2} )}  \left | \frac{\partial \tilde{h}}{\partial z_1} \right |^2 \geq -C_5. 
\end{align*}Combining (\ref{eq:3.24}) and (\ref{eq:3.30}), we get
\begin{align*}
\label{eq:3.31}
\mathcal{L}_F \left (  G( \tilde{\Lambda} )   \right )
\geq   -C_6  -\frac{1}{(1+ \tilde{\lambda}_1^2)(C_{\epsilon_0} + \tilde{\lambda}_1)^2} \left |   \tensor[]{\hat{\chi}}{_{1 \bar{1},  {1}}} \right |^2. \tag{3.31}
\end{align*}By Cauchy--Schwartz inequality, we can get the third derivative term $u_{1 \bar{1} 1}$, which is crucial in our whole estimate. We have
\begin{align*}
\label{eq:3.32}
\mathcal{L}_F \left (  G( \tilde{\Lambda} )   \right )
&\geq  - C_6 -\frac{1}{(1+ \tilde{\lambda}_1^2)(C_{\epsilon_0} + \tilde{\lambda}_1)^2} \left |   \tensor[]{ {\chi}}{_{1 \bar{1},  {1}}} + u_{1 \bar{1} 1} \right |^2 \geq - C_7   -  \frac{(1 + \epsilon_1)\left | u_{1 \bar{1} 1} \right|^2}{(1 + \tilde{\lambda}_1^2)(C_{\epsilon_0} + \tilde{\lambda}_1)^2} ,  \tag{3.32}
\end{align*}where $\epsilon_1$ will be determined later.\bigskip

Let us move back to our locally defined test function $U(x)$. The elliptic operator $\mathcal{L}_F$ acting on other terms will yield
\begin{align*}
\label{eq:3.33}
\mathcal{L}_F \left ( u \right ) 
&=  \frac{u_{i \bar{i}}}{1 + \tilde{\lambda}_i^2} =  \frac{-\chi_{i \bar{i}} + \tilde{\lambda}_i + B_{ii}}{1 + \tilde{\lambda}_i^2}, \tag{3.33}
\end{align*}and

\begin{align*}
\label{eq:3.34}
\mathcal{L}_F \left ( \log \left ( 1 - \frac{| \nabla u|^2}{2K} \right ) \right ) &= \frac{-1}{1 + \tilde{\lambda}_i^2} \frac{1}{(-2 K + | \nabla u |^2)^2} \left | u_{ir} u_{\bar{r}} + u_{i \bar{r}} u_r - \omega_{r \bar{s}, i} u_s u_{\bar{r}}   \right |^2  \tag{3.34} \\
&\quad+  \frac{1}{1 + \tilde{\lambda}_i^2} \frac{1}{ ( -2K + |\nabla u |^2 )} \left ( u_{r \bar{i} i} u_{\bar{r}} + u_{\bar{r} \bar{i} i} u_r + u_{r \bar{i}} u_{\bar{r} i} + u_{ri} u_{\bar{r} \bar{i}}   \right)  \\
&\quad- \frac{2}{1 + \tilde{\lambda}_i^2} \frac{1}{ ( -2K + |\nabla u |^2 )} \left ( \Re \left ( \omega_{s \bar{r}, i} u_{r \bar{i}} u_{\bar{s}} \right )   + \Re \left (\omega_{s \bar{r}, \bar{i}} u_{r {i}} u_{\bar{s}} \right)   \right) \\
&\quad+ \frac{1}{1 + \tilde{\lambda}_i^2} \frac{1}{ ( -2K + |\nabla u |^2 )} \left ( \omega_{p \bar{r}, i} \omega_{s \bar{p}, \bar{i}}  + \omega_{p \bar{r}, \bar{i}} \omega_{s \bar{p},  {i}}  - \omega_{s \bar{r}, \bar{i} i}    \right)  u_r u_{\bar{s}} \\
&\geq -C_8 - \frac{1}{1 + \tilde{\lambda}_i^2} \frac{1}{(-2 K + | \nabla u |^2)^2} \left | u_{ir} u_{\bar{r}} + u_{i \bar{r}} u_r - \omega_{r \bar{s}, i} u_s u_{\bar{r}}   \right |^2    \\
&\quad+  \frac{1}{1 + \tilde{\lambda}_i^2} \frac{1}{ ( -2K + |\nabla u |^2 )} \left ( u_{r \bar{i} i} u_{\bar{r}} + u_{\bar{r} \bar{i} i} u_r + u_{r \bar{i}} u_{\bar{r} i} + u_{ri} u_{\bar{r} \bar{i}}   \right)  \\
&\quad- \frac{2}{1 + \tilde{\lambda}_i^2} \frac{1}{ ( -2K + |\nabla u |^2 )} \left ( \Re \left ( \omega_{s \bar{r}, i} u_{r \bar{i}} u_{\bar{s}} \right )   + \Re \left (\omega_{s \bar{r}, \bar{i}} u_{r {i}} u_{\bar{s}} \right)   \right).
\end{align*}


Combining (\ref{eq:3.32}), (\ref{eq:3.33}), and (\ref{eq:3.34}), we get

\begin{align*}
\label{eq:3.35}
 \mathcal{L}_F \left ( U(x)   \right )   
&\geq - C_7 -C_8    - \frac{1 + \epsilon_1}{(1 + \tilde{\lambda}_1^2)(C_{\epsilon_0} + \tilde{\lambda}_1)^2} \left | u_{1 \bar{1} 1} \right|^2 + A   \frac{\chi_{i \bar{i}} - \tilde{\lambda}_i - B_{ii}}{1 + \tilde{\lambda}_i^2} \tag{3.35}  \\
&\quad+ \frac{1}{2} \frac{1}{1 + \tilde{\lambda}_i^2} \frac{1}{(-2 K + | \nabla u |^2)^2} \left | u_{ir} u_{\bar{r}} + u_{i \bar{r}} u_r - \omega_{r \bar{s}, i} u_s u_{\bar{r}}   \right |^2 \\
&\quad- \frac{1}{2} \frac{1}{1 + \tilde{\lambda}_i^2} \frac{1}{ ( -2K + |\nabla u |^2 )} \left ( u_{r \bar{i} i} u_{\bar{r}} + u_{\bar{r} \bar{i} i} u_r    \right) + \frac{1}{2} \frac{1}{1 + \tilde{\lambda}_i^2} \frac{|u_{r \bar{i}}|^2 + |u_{ri}|^2}{ 2K - |\nabla u |^2 }  \\
&\quad+  \frac{1}{1 + \tilde{\lambda}_i^2} \frac{1}{ ( -2K + |\nabla u |^2 )} \left ( \Re \left ( \omega_{s \bar{r}, i} u_{r \bar{i}} u_{\bar{s}} \right )   + \Re \left (\omega_{s \bar{r}, \bar{i}} u_{r {i}} u_{\bar{s}} \right)  \right).
\end{align*}

Also, at the maximum point of $G( \tilde{\Lambda} ) - \frac{1}{2}  \log \left( 1 - \frac{| \nabla u |^2}{2 K}  \right )  - A u$, the partial derivative will be $0$, which implies that $\frac{\partial G}{\partial \Lambda_{r}^s} \left (  \tilde{\Lambda}  \right ) \frac{\partial \tilde{\Lambda}_{r}^s}{\partial z_i} + \frac{1}{4K - 2|\nabla u|^2} \left ( - \omega_{r \bar{s}, i} u_s u_{\bar{r}} + u_{ri} u_{\bar{r}} + u_r u_{\bar{r} i} \right ) - A u_i   = 0$. Once we simplify it, we get $\frac{1}{C_{\epsilon_0} + \tilde{\lambda}_1} \left ( \frac{\partial \chi_{1 \bar{1}}}{\partial z_i}   + u_{1 \bar{1} i}  \right )  + \frac{1}{4K - 2 |\nabla u|^2} \left ( - \omega_{r \bar{s}, i} u_s u_{\bar{r}} + u_{ri} u_{\bar{r}} + u_r u_{\bar{r} i} \right ) - A u_i  =0$, which implies that
\begin{align*}
\label{eq:3.36}
 u_{1 \bar{1} i} =  & - \frac{\partial \chi_{1 \bar{1}}}{\partial z_i} + A ( C_{\epsilon_0} + \tilde{\lambda}_1) u_i  - \frac{C_{\epsilon_0} + \tilde{\lambda}_1}{4K - 2 |\nabla u|^2} \left ( - \omega_{r \bar{s}, i} u_s u_{\bar{r}} + u_{ri} u_{\bar{r}} + u_r u_{\bar{r} i} \right ).  \tag{3.36}
\end{align*}

By applying Cauchy--Schwartz inequality to (\ref{eq:3.36}), we get
\begin{align*}
\label{eq:3.37}
|u_{1 \bar{1} 1}|^2 &=\left|  - \frac{\partial \chi_{1 \bar{1}}}{\partial z_1} - \frac{C_{\epsilon_0} + \tilde{\lambda}_1}{4K - 2 |\nabla u|^2} \left ( - \omega_{r \bar{s}, i} u_s u_{\bar{r}} + u_{r1} u_{\bar{r}} + u_r u_{\bar{r} 1} \right ) + A (C_{\epsilon_0} + \tilde{\lambda}_1) u_1   \right|^2 \tag{3.37}  \\
&\leq (1 + \epsilon_2) \left(  \frac{C_{\epsilon_0}  + \tilde{\lambda}_1}{4K - 2 |\nabla u|^2}  \right)^2 \left| - \omega_{r \bar{s}, i} u_s u_{\bar{r}} + u_{r1} u_{\bar{r}} + u_r u_{\bar{r} 1}\right|^2 \\ 
&\qquad \qquad \qquad \qquad \qquad \qquad \qquad \qquad  + C_9  + 2(1+ 1/ \epsilon_2)  A^2( C_{\epsilon_0}  + \tilde{\lambda}_1)^2 K.  
\end{align*}Keep in mind that $K = 1 + \sup_X \left | \nabla u  \right |^2_{\omega}$.\bigskip

On the other hand, by (\ref{eq:3.12}), at the maximum point, we have
\begin{align*}
\label{eq:3.38}
u_{i \bar{i} r} = - \frac{\partial \chi_{i \bar{i}}}{\partial z_r} + (1 + \tilde{\lambda}_i^2) \frac{\partial \tilde{h}}{\partial z_r}. \tag{3.38}
\end{align*}

Furthermore, by \hyperlink{P:2.1}{Proposition 2.1}, we may assume that $\tilde{\lambda}_1 = \lambda_1$ is sufficiently large, otherwise we get the $C^2$ estimate and we are done. Thus $\tilde{\lambda}_i$ satisfy the inequality $\sum_i \frac{\chi_{i \bar{i}} - \tilde{\lambda}_i }{1 + \tilde{\lambda}_i^2} \geq \kappa \sum_i \frac{1}{1 + \tilde{\lambda}_i^2}$. By our choice of the constant matrix $B$, we get $\sum_i \frac{\chi_{i \bar{i}} - \tilde{\lambda}_i - B_{ii}}{1 + \tilde{\lambda}_i^2} \geq \frac{\kappa}{2} \sum_i \frac{1}{1 + \tilde{\lambda}_i^2}$. Then we may choose $A > 0$ sufficiently large such that 
\begin{align*}
\label{eq:3.39}
A \frac{\kappa}{1+ \tilde{\lambda}_n^2} \geq 2 \left (C_7 + C_8 +  C_9 +   C_{10} +  C_{11} \right ), \tag{3.39}
\end{align*}where the constants $C_9, C_{10}$ and $C_{11}$ will show up later. Also, if $\tilde{\lambda}_1 = \lambda_1$ is sufficiently large relative to $\chi_{1 \bar{1}}$, then 
\begin{align*}
\label{eq:3.40}
|u_{1 \bar{1}}|^2 \geq \frac{1}{2} \tilde{\lambda}_1^2. \tag{3.40}
\end{align*}

Combining (\ref{eq:3.35}), (\ref{eq:3.37}), and (\ref{eq:3.38}), we get
\begin{align*}
\label{eq:3.41}
 &\mathcal{L}_F \left ( U(x)    \right )  \tag{3.41}  \\ 
&\geq -C_7 -C_8  + \frac{A}{2}   \frac{ \kappa }{1 + \tilde{\lambda}_n^2} - \frac{2 A^2(1 + \epsilon_1)(1+ 1/\epsilon_2)   ) K}{ 1 + \tilde{\lambda}_1^2}  +\frac{1}{2} \frac{1}{1 + \tilde{\lambda}_i^2} \frac{|u_{r \bar{i}}|^2 + |u_{ri}|^2}{ 2K - |\nabla u |^2 }  \\
&\quad+  \frac{2 - (1+ \epsilon_1)(1 + \epsilon_2)}{4(1 + \tilde{\lambda}_i^2)(-2 K + | \nabla u |^2)^2} \left | u_{ir} u_{\bar{r}} + u_{i \bar{r}} u_r - \omega_{r \bar{s}, i} u_s u_{\bar{r}}   \right |^2 \\
&\quad-   \frac{1}{1 + \tilde{\lambda}_i^2} \frac{1}{ ( -2K + |\nabla u |^2 )} \Re \left (   \left ( - \frac{\partial \chi_{i \bar{i}}}{\partial z_{\bar{r}}} + (1 + \tilde{\lambda}_i^2) \frac{\partial h}{\partial z_{\bar{r}}}   \right ) u_r    \right) \\
&\quad+  \frac{1}{1 + \tilde{\lambda}_i^2} \frac{1}{ ( -2K + |\nabla u |^2 )} \left ( \Re \left ( \omega_{s \bar{r}, i} u_{r \bar{i}} u_{\bar{s}} \right )   + \Re \left (\omega_{s \bar{r}, \bar{i}} u_{r {i}} u_{\bar{s}} \right)  \right)      \\
&\geq -C_7 -C_8  + \frac{A}{2}   \frac{ \kappa }{1 + \tilde{\lambda}_n^2} - \frac{8 A^2  K}{ 1 + \tilde{\lambda}_1^2}  + \frac{1}{2} \frac{1}{1 + \tilde{\lambda}_1^2} \frac{|u_{1 \bar{1}}|^2  }{ 2K - |\nabla u |^2 } \\
&\quad+   \frac{1}{1 + \tilde{\lambda}_i^2} \frac{1}{ ( 2K - |\nabla u |^2 )} \Re \left (   \left ( - \frac{\partial \chi_{i \bar{i}}}{\partial z_{\bar{r}}} + (1 + \tilde{\lambda}_i^2) \frac{\partial h}{\partial z_{\bar{r}}}   \right ) u_r    \right) \\
&\quad+  \frac{1}{1 + \tilde{\lambda}_i^2} \frac{1}{ ( -2K + |\nabla u |^2 )} \left ( \Re \left ( \omega_{s \bar{r}, i} u_{r \bar{i}} u_{\bar{s}} \right )   + \Re \left (\omega_{s \bar{r}, \bar{i}} u_{r {i}} u_{\bar{s}} \right)  \right)   + \frac{1}{2} \frac{1}{1 + \tilde{\lambda}_i^2} \frac{   |u_{ri}|^2}{ 2K - |\nabla u |^2 }.          
\end{align*}

The last inequality holds by picking $\epsilon_1 = \frac{1}{3}, \epsilon_2 = \frac{1}{2}$ and dropping off some positive terms.

Note again that by Cauchy--Schwartz inequality, we have the following inequality 

\begin{align*}
\label{eq:3.42}
&\frac{1}{1 + \tilde{\lambda}_i^2} \frac{1}{ ( 2K - |\nabla u |^2 )} \Re \left (   \left ( - \frac{\partial \chi_{i \bar{i}}}{\partial z_{\bar{r}}} + (1 + \tilde{\lambda}_i^2) \frac{\partial h}{\partial z_{\bar{r}}}   \right ) u_r    \right) \tag{3.42} \\
&\geq -\frac{1}{2} \frac{1}{2K - |\nabla u|^2} \left (  \frac{1}{(1 + \tilde{\lambda}_i^2)^2} \left | - \frac{\partial \chi_{i \bar{i}}}{\partial z_r} + (1 + \tilde{\lambda}_i^2) \frac{\partial h}{\partial z_r}  \right |^2 +  \left | u_r \right |^2 \right )   \\
&\geq -\frac{1}{2} \frac{1}{2K - |\nabla u|^2} \left (  \frac{1}{(1 + \tilde{\lambda}_i^2)^2}  C_9 (1 + \tilde{\lambda}_i^2)^2 +  \left | u_r \right |^2 \right )  \geq -\frac{1}{2} \frac{C_9 K}{2K - |\nabla u|^2}  \geq -  C_9. 
\end{align*}
In addition,

\begin{align*}
\label{eq:3.43}
&\frac{1}{1 + \tilde{\lambda}_i^2} \frac{1}{ ( -2K + |\nabla u |^2 )}   \Re \left (\omega_{s \bar{r}, \bar{i}} u_{r {i}} u_{\bar{s}} \right)  \tag{3.43} \\
&\geq  \frac{1}{1 + \tilde{\lambda}_i^2} \frac{1}{ ( -2K + |\nabla u |^2 )}  \left ( \frac{1}{2} \left | u_{ri} \right |^2  + \frac{1}{2} \left |  \omega_{s \bar{r}, \bar{i}} u_{\bar{s}}  \right |^2       \right ) \\
&\geq  \frac{1}{1 + \tilde{\lambda}_i^2} \frac{1}{ ( -2K + |\nabla u |^2 )}  \left ( \frac{1}{2} \left | u_{ri} \right |^2  + C_{10} K      \right )  \geq  \frac{1}{2} \frac{1}{1 + \tilde{\lambda}_i^2} \frac{1}{ ( -2K + |\nabla u |^2 )}    \left | u_{ri} \right |^2   - C_{10},
\end{align*}and

\begin{align*}
\label{eq:3.44}
&\frac{1}{1 + \tilde{\lambda}_i^2} \frac{1}{ ( -2K + |\nabla u |^2 )} \Re \left ( \omega_{s \bar{r}, i} u_{r \bar{i}} u_{\bar{s}} \right ) \tag{3.44} \\
&= \frac{1}{1 + \tilde{\lambda}_i^2} \frac{1}{ ( -2K + |\nabla u |^2 )}  \Re \left ( \omega_{s \bar{r}, i} \left ( \tilde{\lambda}_i \delta_{ri} + B_{ii} \delta_{ri} - \chi_{r \bar{i}}    \right ) u_{\bar{s}} \right ) 
\geq - C_{11}. 
\end{align*}

Eventually, by combining (\ref{eq:3.39}), (\ref{eq:3.40}), (\ref{eq:3.41}), (\ref{eq:3.42}), (\ref{eq:3.43}), and (\ref{eq:3.44}), at the maximum point, we get

\begin{align*}
\label{eq:3.45}
0 &\geq \mathcal{L}_F \left ( U(x)     \right ) \tag{3.45}  \\
&\geq  -C_7 -C_{8} - C_9 - C_{10} - C_{11}    + \frac{A}{2}   \frac{ \kappa }{1 + \tilde{\lambda}_n^2} - \frac{8 A^2   K}{(1 + \tilde{\lambda}_1^2)}  + \frac{1}{2} \frac{1}{1 + \tilde{\lambda}_1^2} \frac{|u_{1 \bar{1}}|^2  }{ 2K - |\nabla u |^2 }  \\
&\geq     -  \frac{8 A^2   K}{1 + \tilde{\lambda}_1^2} +   \frac{\tilde{\lambda}_1^2}{ 8K (1+ \tilde{\lambda}_1^2)}.
\end{align*}At the maximum point, we get $8AK \geq \tilde{\lambda_1} = \lambda_1$. Thus we finish our proof.

\end{proof}

\subsection{\texorpdfstring{$C^1$}{C1} Estimate of Deformed Hermitian--Yang--Mills Equation}
In the following, we give a direct proof for the $C^1$ estimate using the maximal principle instead of the blowup argument. We follow the idea by Collins--Yau \cite{collins2018moment}. Consider the following test function
\begin{align*}
Q(x) \coloneqq \log \left ( 1 + \left | \nabla u \right |^2_\omega   \right ) + \gamma(u), 
\end{align*}where $\gamma(u)$ will be determined later.\bigskip

Let us apply the maximum principle to the above test function $Q(x)$, at the maximum point $x_0$ of $Q(x)$. Similar to the \hyperlink{tilde Lambda}{perturbation} before, we perturb the eigenvalues at the maximum point to make them all distinct. Then we have
\begin{align*}
\label{eq:3.46}
0 \geq \mathcal{L}_F  Q(x_0) &= \mathcal{L}_F \left ( \log \left ( 1 + \left | \nabla u \right |^2_\omega   \right ) + \gamma(u) \right ) \tag{3.46} \\ 
&= \frac{\pa F  }{\pa \tilde{\Lambda}^k_i } \left (  {\Lambda}  \right )  \omega^{k \bar{j}}  \frac{\pa^2}{\pa z_i\pa\bar z_j}  \left ( \log \left ( 1 + \left | \nabla u \right |^2_\omega   \right ) + \gamma(u) \right ) \\
&= \frac{1}{1+ \tilde{\lambda}_i^2} \left (  \gamma''(u) | u_i |^2 + \gamma'(u) u_{i \bar{i}} \right )\\ 
&\quad- \frac{1}{1 +  \tilde{\lambda}_i^2} \frac{1}{( 1 + | \nabla u |^2)^2} \left | \sum_r \left ( u_{ir} u_{\bar{r}} + u_{i \bar{r}} u_r - \sum_s \omega_{r \bar{s}, i} u_s u_{\bar{r}} \right )   \right |^2   \\
&\quad+  \frac{2}{1 +  \tilde{\lambda}_i^2} \frac{1}{ (  1 + |\nabla u |^2 )} \Re \left ( u_{r \bar{i} i} u_{\bar{r}}    \right)  - \frac{1}{1 +  \tilde{\lambda}_i^2} \frac{1}{ (  1 + |\nabla u |^2 )}    \omega_{s \bar{r}, \bar{i} i}     u_r u_{\bar{s}}      \\
&\quad+ \sum_{i, r} \frac{1}{1 +  \tilde{\lambda}_i^2} \frac{1}{ (  1 + |\nabla u |^2 )} \left (  \left | u_{r \bar{i}} - \sum_s \omega_{r \bar{s}, \bar{i}} u_s  \right |^2 + \left | u_{r i} - \sum_s \omega_{r \bar{s}, i} u_s  \right |^2    \right).
\end{align*}To estimate the term (\ref{eq:3.47}), we substitute $S_i \coloneqq \sum_r u_{ir} u_{\bar{r}}, T_i \coloneqq \sum_r u_{i \bar{r}} u_r,$ and $R_i \coloneqq \sum_{r, s} \omega_{r \bar{s}, i} u_s u_{\bar{r}}$ to make it easy to see. Then we have
\begin{align*}
\label{eq:3.47}
&\left | \sum_r \left ( u_{ir} u_{\bar{r}} + u_{i \bar{r}} u_r - \sum_s \omega_{r \bar{s}, i} u_s u_{\bar{r}} \right )  \right |^2 \tag{3.47} \\
&= \left | S_i + T_i - R_i \right |^2 = 2 \Re \left ( (S_i + T_i - R_i) \bar{T_i}   \right ) + |S_i - R_i|^2 - |T_i|^2 \\
&\leq 2 \Re \left ( (S_i + T_i - R_i) \bar{T_i}   \right ) + |S_i - R_i|^2.  
\end{align*}Also, by Cauchy-Schwartz inequality, we have
\begin{align*}
\label{eq:3.48}
|S_i - R_i|^2 &= \left |  \sum_r \left (  u_{ir} u_{\bar{r}} - \sum_s \omega_{r \bar{s},i} u_s  u_{\bar{r}}   \right )    \right |^2 \leq \left (  \sum_r  |u_{\bar{r}}|^2  \right ) \cdot \left (  \sum_r  \left | u_{ir} - \omega_{r \bar{s}} u_s \right |^2  \right ) \tag{3.48} \\
&= \left | \nabla u   \right |^2 \cdot \left (  \sum_r  \left | u_{ir} - \omega_{r \bar{s}} u_s \right |^2  \right ).
\end{align*}

Hence, by dropping off some positive terms and by combing equations (\ref{eq:3.46}), (\ref{eq:3.47}), and (\ref{eq:3.48}), we obtain
\begin{align*}
\label{eq:3.49}
0 &\geq \mathcal{L}_F  Q(x_0) \tag{3.49} \\ 
&\geq \frac{1}{1+ \tilde{\lambda}_i^2} \left (  \gamma''(u) | u_i |^2 + \gamma'(u) u_{i \bar{i}} \right )\\ 
&\quad- \frac{2}{1 +  \tilde{\lambda}_i^2} \frac{1}{( 1 + | \nabla u |^2)^2} \Re \left ( \left (\sum_r \left ( u_{ir} u_{\bar{r}} + u_{i \bar{r}} u_r - \sum_s \omega_{r \bar{s}, i} u_s u_{\bar{r}} \right ) \right ) \cdot \left ( \sum_p u_{p \bar{i}} u_{\bar{p}}   \right )   \right )   \\
&\quad+  \frac{2}{1 +  \tilde{\lambda}_i^2} \frac{1}{ (  1 + |\nabla u |^2 )} \Re \left ( u_{r \bar{i} i} u_{\bar{r}}    \right)  - \frac{1}{1 +  \tilde{\lambda}_i^2} \frac{1}{ (  1 + |\nabla u |^2 )}    \omega_{s \bar{r}, \bar{i} i}     u_r u_{\bar{s}}.      
\end{align*}Now, at the maximum point $x_0$ of $Q(x)$, we have
\begin{align*}
0 = \frac{\partial Q}{\partial z_i} (x_0) &= \frac{1}{1+ |\nabla u|^2} \left (  u_{ir} u_{\bar{r}} + u_{i \bar{r}} u_r - \omega_{r \bar{s}, i} u_s u_{\bar{r}}  \right ) - \gamma'(u) u_i,
\end{align*}which implies that at $x_0$,
\begin{align*}
\label{eq:3.50}
u_{ir} u_{\bar{r}} + u_{i \bar{r}} u_r - \omega_{r \bar{s}, i} u_s u_{\bar{r}} = \gamma'(u) u_i \left ( 1 + | \nabla u |^2 \right ).\tag{3.50}
\end{align*}Also, by the equation (\ref{eq:3.12}), we have
\begin{align*}
u_{i \bar{i}r} &= - \frac{\partial \chi_{i \bar{i}}}{\partial z_r} + (1 + \tilde{\lambda}_i^2) \frac{\partial \tilde{h}}{\partial z_r},
\end{align*}which implies that,
\begin{align*}
\label{eq:3.51}
\frac{2 \Re \left ( u_{r \bar{i} i} u_{\bar{r}}    \right) }{(1 +  \tilde{\lambda}_i^2) (1 + |\nabla u |^2) }  &= \frac{2 \Re \left ( \left (  - \chi_{i\bar{i},r} + (1 + \tilde{\lambda}_i^2)h_r  \right ) u_{\bar{r}}    \right) }{ (1 +  \tilde{\lambda}_i^2)( 1 + |\nabla u |^2 )}  \geq -C_1 \frac{|\nabla u|}{1 + |\nabla u|^2}. \tag{3.51}
\end{align*}For the last term, we have
\begin{align*}
\label{eq:3.52}
- \frac{\omega_{s \bar{r}, \bar{i} i}     u_r u_{\bar{s}} }{(1 +  \tilde{\lambda}_i^2) (  1 + |\nabla u |^2 )}   &\geq   -C_2. \tag{3.52}
\end{align*}


By combining equations (\ref{eq:3.49}), (\ref{eq:3.50}), (\ref{eq:3.51}), and (\ref{eq:3.52}), and dropping off some positive terms, we get
\begin{align*}
\label{eq:3.53}
0 &\geq -C_2 - C_1 \frac{|\nabla u|}{1 + |\nabla u|^2}  +  \frac{  \gamma''(u) | u_i |^2 + \gamma'(u) u_{i \bar{i}} }{1 + \tilde{\lambda}_i^2}  - \frac{2 \gamma'(u) \Re \left (  u_i u_{p \bar{i}} u_{\bar{p}}   \right )}{(1+ \tilde{\lambda}_i^2)(1+ |\nabla u|^2)} \tag{3.53} \\
&\geq  -C_2 - C_1 \frac{|\nabla u|}{1 + |\nabla u|^2}  +  \frac{  \gamma''(u) | u_i |^2 + \gamma'(u) u_{i \bar{i}} }{1 + \tilde{\lambda}_i^2}  - \frac{2 \gamma'(u) \Re \left (  u_i u_{p \bar{i}} u_{\bar{p}}   \right )}{(1+ \tilde{\lambda}_i^2)(1+ |\nabla u|^2) }. 
\end{align*}

Following Phong--Sturm \cite{phong2009dirichlet, phong2010regularity}, we pick $\gamma(u) \coloneqq -Bu + \frac{1}{ u -\inf_X u +1}$. The rest follows directly from Collins--Yau \cite{collins2018moment}.

\hypertarget{S:4}{\section{Existence Theorem for dHYM Equation}}

We consider the following constant 
\begin{align*}
\label{eq:4.1}
\hat{\Theta}_\omega \left ( \chi  \right ) &= \Arg \left (  \int_X \left (  \omega + \sqrt{-1} \chi \right )^n   \right ), \tag{4.1}
\end{align*}where $\chi \in [\chi_0 ] \in H^{1,1}_{\text{BC}}(X; \mathbb{R})$. Note that we need to specify the branch cut such that $\hat{\Theta}_\omega \left ( 0 \right) = 0$.

\hypertarget{T:4.1}{\begin{ftheorem}[Existence Theorem]}
Assume the Hermitian metric $\omega$ satisfies $\partial \bar{\partial} \omega = 0 = \partial \bar{\partial} \left ( \omega^2 \right )$. Also, suppose that there exists a $C$-subsolution $\ubar{\chi} \coloneqq \chi_0 + \sqrt{-1} \partial \bar{\partial} \ubar{u}$ such that 
\begin{align*}
 \Theta_{\omega} \left ( \ubar{\chi} \right ) > (n-2)\frac{\pi}{2}.
\end{align*}Then there exists a unique smooth $(1,1)$-form $\chi \in [\chi_0]$ solving the deformed Hermitian--Yang--Mills equation
\begin{align*}
\Theta_\omega \left ( \chi \right ) = \hat{\Theta}_\omega \left ( \ubar{\chi} \right ) = \hat{\Theta}_\omega \left ( \chi_0 \right ).
\end{align*}
\end{ftheorem}

\hypertarget{L:4.1}{\begin{flemma}}
Assume the Hermitian metric $\omega$ satisfies $\partial \bar{\partial} \omega = 0 = \partial \bar{\partial} \left ( \omega^2 \right )$. Then for any closed real (1,1)-form $\chi \in [\chi_0]$, we have
\begin{align*}
\label{eq:4.2}
\hat{\Theta}_\omega \left ( \chi \right ) = \hat{\Theta}_\omega \left ( \chi_0 \right ). \tag{4.2}
\end{align*}
\end{flemma}
\begin{proof}
Now, since $\chi \in [ \chi_0 ]$, we can write $\chi = \chi_0 + \sqrt{-1} \partial \bar{\partial} v$. Thus
\begin{align*}
\hat{\Theta}_\omega \left (   \chi \right ) &= \Arg \int_X \left (  \omega + \sqrt{-1} \chi   \right )^n = \Arg \int_X \left (  \left ( \omega + \sqrt{-1} \chi_0 \right ) - \partial \bar{\partial} v   \right )^n \\
&= \Arg  \int_X  \sum_{i =0}^n (-1)^i  \binom{n}{i} \left (  \partial \bar{\partial} v    \right )^i \wedge  \left (  \omega + \sqrt{-1} \chi_0   \right )^{n-i} \\
&= \Arg  \int_X   \left (  \left (  \omega + \sqrt{-1} \chi_0   \right )^{n} +   \sum_{i =1}^{n} (-1)^i  \binom{n}{i}  \left (  \partial \bar{\partial} v    \right )^i \wedge  \left (  \omega + \sqrt{-1} \chi_0   \right )^{n-i}   \right ).
\end{align*}By Stoke's Theorem, for $1 \leq i \leq n$, we have

\begin{align*}
&\int_X  \left (  \partial \bar{\partial} v    \right )^i \wedge  \left (  \omega + \sqrt{-1} \chi_0   \right )^{n-i}   \\
&=  \int_X   \bar{\partial} v \wedge  \left (  \partial \bar{\partial} v    \right )^{i-1} \wedge  \partial \left ( \left (  \omega + \sqrt{-1} \chi_0   \right )^{n-i}    \right   )       \\
&= \int_X     v   \left (  \partial \bar{\partial} v    \right )^{i-1} \wedge  \partial \bar{\partial} \left ( \left (  \omega + \sqrt{-1} \chi_0   \right )^{n-i}    \right   )   \\
&= \int_X      v   \left (  \partial \bar{\partial} v    \right )^{i-1} \wedge   \frac{(n-i) (n-1-i)}{2}  \partial \bar{\partial} \left ( \omega^2 \right )   \wedge  \left (  \omega + \sqrt{-1} \chi_0   \right )^{n - 2 - i}   = 0. 
\end{align*}In summary, we have
\begin{align*}
\hat{\Theta}_\omega \left ( \chi \right ) &= \Arg \int_X \left (  \omega + \sqrt{-1} \chi \right )^n = \Arg \int_X \left (  \omega + \sqrt{-1} \chi_0 \right )^n.
\end{align*}

\end{proof}

\hypertarget{L:4.2}{\begin{flemma}[\cite{collins20151}]}
Fix $k \geq 2, \gamma \in (0, 1)$ and suppose that we have $C^{k-2, \gamma}$ functions $H_0, H_1\colon X \rightarrow \mathbb{R}$ and a $C^{k, \gamma}$ function $u\colon X \rightarrow \mathbb{B}$ such that
\begin{align*}
\Theta_\omega \left (  \chi_0 + \sqrt{-1} \partial \bar{\partial} u \right ) &= H_0.
\end{align*}Consider the family of equations 
\begin{align*}
\label{eq:4.3}
\Theta_\omega \left ( \chi_0 + \sqrt{-1} \partial \bar{\partial} u_t \right ) &= (1- t) H_0 + t H_1 + c_t, \tag{4.3}
\end{align*}where $c_t$ is a constant. Then there exists $\epsilon > 0$ such that, for any $|t| < \epsilon$, there exists an unique pair $(u_t, c_t) \in C^{k, \gamma} \times \mathbb{R}$ solving (\ref{eq:4.3}). Furthermore, if $H_0, H_1$ are smooth, then so is $u_t$.
\end{flemma}

Suppose now that we have a Hermitian metric $\omega$ satisfying the assumptions of \hyperlink{T:4.1}{Theorem 4.1} and a $C$-subsolution $\ubar{\chi}$ to the dHYM equation satisfying the assumptions of \hyperlink{T:4.1}{Theorem 4.1}. Let us define
\begin{align*}
\Theta_0 \coloneqq \Theta_\omega \left ( \ubar{\chi}  \right ).
\end{align*}Now, without loss of generality, we may assume that $\Theta_0 \neq \hat{\Theta}_\omega \left ( \ubar{\chi}  \right )$, otherwise we are done. Let $\mu_1, \dots, \mu_n$ be the eigenvalues of the Hermitian endomorphism $\omega^{-1} \ubar{\chi}$ at an arbitrary point of $X$. We have
\begin{align*}
\sum_{l \neq j} \arctan \left ( \mu_l \right ) > \Theta_0 - \frac{\pi}{2}.
\end{align*}Thus by \hyperlink{L:2.4}{Lemma 2.4}, $\ubar{\chi}$ is also a $C$-subsolution of $\sum_i \arctan(\lambda_i) = \Theta_0$. By \hyperlink{L:2.5}{Lemma 2.5, (ii)}, we see that $\ubar{\chi}$ is a $C$-subsolution of $\max \left \{ \hat{\Theta}_\omega \left ( \ubar{\chi}  \right ), \Theta_0   \right \}$. So, we can find $\delta > 0$ sufficiently small such that 
\begin{align*}
\sum_{l \neq j} \arctan \left (  \mu_l   \right ) > \max \left \{  \Theta_0, \hat{\Theta}_\omega \left (  \ubar{\chi}  \right )   \right \}  + 2 \delta - \frac{\pi}{2};\quad
\Theta_0 > (n-2) \frac{\pi}{2} + 2 \delta ;\quad
\left ( \hat{\Theta}_\omega \left ( \ubar{\chi} \right ) - \inf_X \Theta_0  \right ) > 4 \delta.
\end{align*}

By \hyperlink{L:2.5}{Lemma 2.5, (i)}, $C$-subsolution is an open condition, so we can consider a smooth function \hypertarget{T1}{$\Theta_1$} near the continuous function $ \max \left \{  \Theta_0, \hat{\Theta}_\omega \left (  \ubar{\chi}  \right )   \right \}$, where $\ubar{\chi}$ is again a $C$-subsolution of $\Theta_1$, satisfying 
\begin{itemize}
\item[1.] $\max \left \{ \Theta_0, \hat{\Theta}_\omega \left ( \ubar{\chi} \right )   \right \} \leq \Theta_1 \leq \max \left \{ \Theta_0, \hat{\Theta}_\omega \left ( \ubar{\chi} \right )   \right \} + \delta$.
\item[2.] Fix a point $p \in X$ where $\Theta_0$ achieves its minimum, we require that $\Theta_1 \equiv \hat{\Theta}_\omega \left ( \ubar{\chi} \right )$ in a neighborhood of $p \in X$.
\end{itemize}

Now, we use the function $\Theta_1$ as the first target for the method of continuity. One can see that the function $\Theta_1$ here is quite flexible. In Collins--Jacob--Yau \cite{collins20151}, they construct an explicit $\Theta_1$ by regularizing the maximum function $\max \left \{ \Theta_0, \hat{\Theta}_\omega \left ( \ubar{\chi} \right ) \right \}$; one can find the trick in Demailly \cite{demailly1997complex}.

\hypertarget{P:4.1}{\begin{fprop}}There exists a smooth function $u_1\colon X \rightarrow \mathbb{R}$ and a constant $b_1 < 0$ such that
\begin{align*}
\Theta_\omega \left (   \ubar{\chi} + \sqrt{-1} \partial \bar{\partial} u_1 \right ) = \Theta_1 + b_1 \text{ and  } \  \Theta_1 + b_1 > (n-2) \frac{\pi}{2}.
\end{align*}
\end{fprop}

\begin{proof}
We use the method of continuity. Consider the family of equations 
\begin{align*}
\label{eq:4.4}
\Theta_\omega \left ( \ubar{\chi} + \sqrt{-1} \partial \bar{\partial} u_t \right ) = (1-t) \Theta_0 + t \Theta_1 + b_t. \tag{4.4}
\end{align*}Define $I \coloneqq \left \{  t \in [0, 1]: \exists \  (u_t, b_t)  \in C^{\infty}(X) \times \mathbb{R} \text{ solving (\ref{eq:4.4})}   \right \}$.
As usual, we try to show that $I$ is non-empty, open and closed. Then we get $I = [0, 1]$, which proves our \hyperlink{P:4.1}{Proposition 4.1}.

\begin{itemize}
\item $I$ is non-empty:\\
It is straightforward to see that $(0, 0) \in C^{\infty}(X) \times \mathbb{R}$ is a solution at time $t = 0$, so we have that $I$ is non-empty.
\item $I$ is open:\\
By \hyperlink{L:4.2}{Lemma 4.2}, the set $I$ is open.
\item $I$ is closed:\\
The closedness condition follows from the a priori estimates in \hyperlink{T:1.1}{Theorem 1.1} with a standard bootstrapping argument provided we can show 
\begin{itemize}
\item[1.] $\ubar{\chi}$ is a subsolution of equation (\ref{eq:4.4}) for all $t \in [0, 1]$.
\item[2.] $(1-t) \Theta_0 + t \Theta_1 + b_t > (n-2)\frac{\pi}{2}$ uniformly for $t \in [0, 1]$.
\end{itemize}
\end{itemize}
To prove the first point, recall \hyperlink{L:2.4}{Lemma 2.4}. We try to prove that
\begin{align*}
\sum_{i \neq j} \arctan \left (  \mu_i  \right ) > (1 - t) \Theta_0 + t \Theta_1 + b_t - \frac{\pi}{2},\quad  \forall j.
\end{align*}First, we claim that $b_t \leq t \sup_X \left ( \Theta_0 - \Theta_1  \right ) \leq 0$. By picking $q \in X$ such that $u_t$ achieves its maximum, the ellipticity implies
\begin{align*}
\Theta_0 (q) = \Theta_\alpha \left (   \ubar{\chi} \right ) (q) \geq   \Theta_\alpha \left (   \ubar{\chi}  + \sqrt{-1} \partial \bar{\partial} u_t \right ) (q) = (1-t) \Theta_0 (q) + t \Theta_1(q) + b_t.
\end{align*}Hence, $b_t \leq t \left (  \Theta_0 (q) - \Theta_1 (q)   \right ) \leq t \sup_X \left (  \Theta_0  - \Theta_1    \right ) \leq 0$, where the last inequality follows by the choice of \hyperlink{T1}{$\Theta_1$}.

Therefore, we have
\begin{align*}
\sum_{i \neq j} \arctan \left (  \mu_i   \right ) &> \max \left \{ \Theta_0, \hat{\Theta}_\omega \left ( \ubar{\chi} \right ) \right \} + 2 \delta - \frac{\pi}{2} > \Theta_1 - \frac{\pi}{2} \geq (1-t) \Theta_0 + t \Theta_1 + b_t - \frac{\pi}{2}.
\end{align*}So by \hyperlink{L:2.4}{Lemma 2.4}, $\ubar{\chi}$ is a subsolution of equation (\ref{eq:4.4}) for all $t \in [0, 1]$.\bigskip

Second, similarly, at a point where $u_t$ achieves its minimum, we find $b_t \geq -t \sup_X \left ( \Theta_1 - \Theta_0 \right )$. Note that with the second requirement of \hyperlink{T1}{$\Theta_1$}, we have
\begin{align*}
\label{eq:4.5}
(1-t) \Theta_0 (p) + t \hat{\Theta}_\omega \left ( \ubar{\chi} \right ) &= (1-t) \Theta_0 (p) + t \Theta_1(p) \geq \inf_X \left \{ (1-t)\Theta_0 + t \Theta_1  \right \} \tag{4.5} \\
&\geq (1-t) \inf_X \Theta_0 + t \inf_X \Theta_1   = (1-t)\Theta_0 (p) + t \hat{\Theta}_\omega \left ( \ubar{\chi} \right ).
\end{align*}Also, we claim that $\Theta_1 (p) - \Theta_0 (p) \geq \sup_X \left (  \Theta_1 - \Theta_0   \right ) - \delta$. To show this claim, first, we consider the set $U_1 \coloneqq \left \{ x \in X : \Theta_0 + \delta \leq \hat{\Theta}_\omega\left ( \ubar{\chi} \right ) - \delta \right \}$, on which we have $\hat{\Theta}_\omega\left ( \ubar{\chi} \right ) - \Theta_0 \leq \Theta_1 - \Theta_0 \leq \hat{\Theta}_\omega \left ( \ubar{\chi} \right ) - \Theta_0 + \delta$. On the set $U_2 \coloneqq \left \{ x \in X : \hat{\Theta}_\omega\left ( \ubar{\chi} \right ) + \delta \leq  \Theta_0  - \delta \right \}$, we have $0 \leq \Theta_1 - \Theta_0 \leq \delta$. Finally, on the set $U_3 \coloneqq \left \{ x \in X : | \Theta_0 - \hat{\Theta}_\omega\left ( \ubar{\chi} \right ) | < 2 \delta \right \}$, we have $\Theta_1 - \Theta_0 \leq \max \left \{ \Theta_0, \hat{\Theta}_\omega \left ( \ubar{\chi} \right )   \right \} + \delta - \Theta_0 \leq 3\delta$. Thus, we see
\begin{align*}
\label{eq:4.6}
\Theta_1 (p) - \Theta_0 (p) = \hat{\Theta}_\omega - \Theta_0(p) \geq \sup_X \left ( \Theta_1 - \Theta_0 \right ) - \delta. \tag{4.6}
\end{align*}

To summarize, by (\ref{eq:4.5}) and (\ref{eq:4.6}), we have
\begin{align*}
\label{eq:4.7}
\inf_X \left [  (1-t) \Theta_0 + t \Theta_1 + b_t  \right ] &= (1-t) \Theta_0 (p) + t \Theta_1 (p) + b_t = \Theta_0 (p) + t \left ( \Theta_1 (p) - \Theta_0 (p) \right ) + b_t \tag{4.7} \\
&\geq  \Theta_0 (p) + t \sup_X \left ( \Theta_1  - \Theta_0  \right ) + b_t - t \delta \geq \Theta_0 (p) - \delta  > (n-2) \frac{\pi}{2}. 
\end{align*}Thus, we see $(1-t) \Theta_0 + t \Theta_1 + b_t > (n-2) \frac{\pi}{2}$ uniformly for $t \in [0, 1]$. By \hyperlink{T:1.1}{Theorem 1.1} together with the usual Schauder estimates and bootstrapping argument, we can conclude that $I$ is closed. 

\end{proof}

\begin{proof}[Proof of \hyperlink{T:4.1}{Theorem 4.1}:] Let $\chi_1 \coloneqq \ubar{\chi} + \sqrt{-1} \partial \bar{\partial} u_1$, where $u_1$ is the function from \hyperlink{P:4.1}{Proposition 4.1}. Again, we consider the method of continuity, 
\begin{align*}
\label{eq:4.8}
\Theta_\omega \left ( \chi_1 + \sqrt{-1} \partial \bar{\partial}  v_t    \right ) &= (1 - t)\Theta_1 + t \hat{\Theta}_\omega \left ( \ubar{\chi} \right ) + c_t. \tag{4.8}
\end{align*}Define $J \coloneqq \left \{  t \in [0, 1]: \exists \  (v_t, c_t)  \in C^{\infty}(X) \times \mathbb{R} \text{ solving (\ref{eq:4.8})}   \right \}$. Same as before, we can see $J$ is non-empty and open. For the closedness part, we try to show that
\begin{itemize}
\item[1.] $\ubar{\chi}$ is a subsolution of equation (\ref{eq:4.8}) for all $t \in [0, 1]$.
\item[2.] $(1-t) \Theta_1 + t \hat{\Theta}_{\omega}\left ( \ubar{\chi}  \right ) + c_t > (n-2)\frac{\pi}{2}$ uniformly for $t \in [0, 1]$.
\end{itemize}As in the proof of \hyperlink{P:4.1}{Proposition 4.1}, it suffices to control $c_t$. By the definition of $\hat{\Theta}_{\omega} \left ( \ubar{\chi} \right )$ and \hyperlink{L:4.1}{Lemma 4.1}, for any $t \in [0, 1]$ we get
\begin{align*}
\label{eq:4.9}
\hat{\Theta}_\omega \left (  \ubar{\chi}  \right ) &= \Arg \int_X  \left ( \omega + \sqrt{-1} \ubar{\chi}   \right )^n  = \Arg \int_X  \left ( \omega + \sqrt{-1} \left (  \chi_1 + \sqrt{-1} \partial \bar{\partial} v_t  \right )   \right )^n  \tag{4.9} \\
&= \Arg \int_X  \sqrt{\frac{\det \eta_t}{\det \omega}} e^{i  \left (  \Theta_\omega \left ( \chi_1 + \sqrt{-1} \partial \bar{\partial}  v_t    \right )  \right ) } \omega^n \\
&= \Arg \left ( \int_X  \sqrt{\frac{\det \eta_t}{\det \omega}} e^{i \left (  (1-t)\Theta_1 + t \hat{\Theta}_{\omega} \left ( \ubar{\chi} \right ) + c_t   \right )} \omega^n \Big\slash \int_X \omega^n \right ),
\end{align*}where we set $\eta_t = \omega + \left ( \chi_1 + \sqrt{-1} \partial \bar{\partial} v_t \right ) \omega^{-1} \left ( \chi_1 + \sqrt{-1} \partial \bar{\partial} v_t \right )$.\bigskip

Since $\Theta_1 \geq \hat{\Theta}_\omega \left ( \ubar{\chi} \right )$, if $c_t > 0$, then for equation (\ref{eq:4.9}) we have
\begin{align*}
\hat{\Theta}_\omega \left (  \ubar{\chi}  \right ) 
&= \Arg \left ( \int_X  \sqrt{\frac{\det \eta_t}{\det \omega}} e^{i \left (  (1-t)\Theta_1 + t \hat{\Theta}_{\omega} \left ( \ubar{\chi} \right ) + c_t   \right )} \omega^n  \Big\slash \int_X \omega^n \right ) \\
&=  \int_X  \Arg \left (  \sqrt{\frac{\det \eta_t}{\det \omega}} e^{i \left (  (1-t)\Theta_1 + t \hat{\Theta}_{\omega} \left ( \ubar{\chi} \right ) + c_t   \right )} \omega^n \right ) \Big\slash \int_X \omega^n \\
&\geq  \int_X  \Arg \left (   e^{i \left (   \hat{\Theta}_{\omega} \left ( \ubar{\chi} \right ) + c_t   \right )} \omega^n \right ) \Big\slash \int_X \omega^n  >   \int_X  \Arg \left (   e^{i    \hat{\Theta}_{\omega} \left ( \ubar{\chi} \right )    } \omega^n \right ) \Big\slash \int_X \omega^n =   \hat{\Theta}_\omega \left (  \ubar{\chi}  \right ),
\end{align*}which gives us a contradiction. So, we have $c_t \leq 0$ for all $t \in [0, 1]$. Thus,
\begin{align*}
\sum_{i \neq j} \arctan \left (  \mu_i   \right ) &>   \Theta_1 - \frac{\pi}{2} \geq (1-t) \Theta_1 + t \hat{\Theta}_\omega \left ( \ubar{\chi} \right ) + c_t - \frac{\pi}{2}.
\end{align*}So we can conclude that $\ubar{\chi}$ is a subsolution of equation (\ref{eq:4.8}) for all $t \in [0, 1]$.\bigskip

Furthermore, for the fixed point $p \in X$ where $\Theta_0$ achieves minimum, our second requirement of \hyperlink{T1}{$\Theta_1$} combined with \hyperlink{P:4.1}{Proposition 4.1} implies $\hat{\Theta}_\omega \left ( \ubar{\chi} \right ) + b_1 = \Theta_1 (p) + b_1 > (n-2) \frac{\pi}{2}$. In particular, we get
\begin{align*}
\label{eq:4.10}
(1-t) \left [  \Theta_1 + b_1   \right ] + t \left [  \hat{\Theta}_{\omega} \left ( \ubar{\chi} \right ) + b_1  \right ] \geq  \hat{\Theta}_{\omega} \left ( \ubar{\chi} \right ) + b_1 > (n-2)\frac{\pi}{2}. \tag{4.10}
\end{align*}In order to show that $(1-t)\Theta_1 + t \hat{\Theta}_\omega \left ( \ubar{\chi} \right ) + c_t > (n-2) \frac{\pi}{2}$ uniformly, it suffices to show that $c_t \geq b_1$ for all $t$. If the minimum of $v_t$ is achieved at the point $q \in X$, then we have
\begin{align*}
\label{eq:4.11}
\Theta_1(q) + b_1 = \Theta_\omega \left (  \chi_1  \right )(q) \leq \Theta_\omega \left (  \chi_1 + \sqrt{-1} \partial \bar{\partial} v_t  \right ) (q) =    (1-t) \Theta_1 (q) + t \hat{\Theta}_\omega ( \ubar{\chi} ) + c_t. \tag{4.11}
\end{align*}Hence by rearranging equation (\ref{eq:4.11}), we get $c_t \geq b_1 + t \left [ \Theta_1 (q) - \hat{\Theta}_\omega ( \ubar{\chi} )    \right ] \geq b_1$. As a result, we can apply the a priori estimates in \hyperlink{T:1.1}{Theorem 1.1} uniformly in $t$ to conclude that $I$ is closed. The higher regularity follows in the usual way from the Schauder estimates and bootstrapping.\bigskip

By \hyperlink{L:4.1}{Lemma 4.1}, we can see that $c_1 = 0$, since
\begin{align*}
\hat{\Theta}_\omega \left (  \ubar{\chi}  \right ) &= \Arg \left ( \int_X  \sqrt{\frac{\det \eta_1}{\det \omega}} e^{i \left (   \hat{\Theta}_{\omega} \left ( \ubar{\chi} \right ) + c_1  \right )} \omega^n  \Big\slash \int_X \omega^n \right ) \\ 
&=  \int_X  \Arg \left (  \sqrt{\frac{\det \eta_1}{\det \omega}} e^{i \left (   \hat{\Theta}_{\omega} \left ( \ubar{\chi} \right ) + c_1   \right )} \omega^n \right ) \Big\slash \int_X \omega^n \\
&=  \int_X  \Arg \left (   e^{i \left (   \hat{\Theta}_{\omega} \left ( \ubar{\chi} \right ) + c_1   \right )} \omega^n \right ) \Big\slash \int_X \omega^n =   \hat{\Theta}_\omega \left (  \ubar{\chi}  \right ) + c_1.
\end{align*}

Therefore, there exists a smooth closed $(1,1)$-form $\chi \coloneqq \chi_1 + \sqrt{-1} \partial \bar{\partial} v_1 \in [ \chi_0 ]$ solving the dHYM equation
\begin{align*}
\Theta_\omega \left ( \chi \right ) &= \hat{\Theta}_\omega \left ( \ubar{\chi} \right ) + c_1= \hat{\Theta}_\omega \left ( \ubar{\chi} \right ).
\end{align*}

Now, to prove the uniqueness, let $\chi, \tilde{\chi} \in [\chi_0]$ be such that $\Theta_\omega \left ( \chi \right ) =   \hat{\Theta}_\omega \left ( \ubar{\chi} \right ) = \Theta_\omega \left ( \tilde{\chi} \right )$. Say $\chi \coloneqq \ubar{\chi} + \sqrt{-1} \partial \bar{\partial} u$ and  $\tilde{\chi} \coloneqq \ubar{\chi} + \sqrt{-1} \partial \bar{\partial} v$. By \hyperlink{L:3.3}{Lemma 3.3}, we get
\begin{align*}
\label{eq:4.12}
0 &= \int_{0}^1 \frac{d}{dt} \Theta_\omega \left (  t\chi + (1-t) \tilde{\chi}  \right ) dt  = \int_{0}^1 \left ( \left (  \Id + \Lambda_t^2 \right )^{-1}  \right )^i_k \omega^{k \bar{p}} (u-v)_{i \bar{p}}  dt \tag{4.12}  \\
&= \int_{0}^1 \left ( \left (  \Id + \Lambda_t^2 \right )^{-1}  \right )^i_k  dt \cdot \omega^{k \bar{p}} (u-v)_{i \bar{p}}  
\end{align*}where $\Lambda_t \coloneqq \omega^{-1} \left ( t\chi + (1-t) \tilde{\chi} \right )$. We can view (\ref{eq:4.12}) as an uniformly elliptic operator. Thus by strong maximum principle, we get $u - v$ is a constant. Hence $\chi = \tilde{\chi}$.

\end{proof}

As a corollary, we have the following result when complex dimension equals $2$.

\hypertarget{C:4.1}{\begin{fcor}}
Let $X$ be a compact complex surface equipped with a Hermitian metric $\omega$. Suppose that there exists a $C$-subsolution $\ubar{\chi} \coloneqq \chi_0 + \sqrt{-1} \partial \bar{\partial} \ubar{u}$ such that 
\begin{align*}
 \Theta_{\omega} \left ( \ubar{\chi} \right ) > 0.
\end{align*}Then there exists a pluriclosed Hermitian metric $\tilde{\omega}$ in the conformal class of $\omega$ and a unique smooth $(1,1)$-form $\chi \in [\chi_0]$ solving the deformed Hermitian--Yang--Mills equation
\begin{align*}
\Theta_{\tilde{\omega}} \left ( \chi \right ) = \hat{\Theta}_{\tilde{\omega}} \left ( \ubar{\chi} \right ) = \hat{\Theta}_{\tilde{\omega}} \left ( \chi_0 \right ).
\end{align*}
\end{fcor}

\begin{proof}
Here, to find such a Hermitian metric $\tilde{\omega}$, by Gauduchon \cite{gauduchon1977theoreme}, there exists a function $g \colon X \rightarrow \mathbb{R}$ such that the Hermitian metric $\tilde{\omega}\coloneqq e^g \omega$ satisfies 
\begin{align*}
\partial \bar{\partial} \left (   \tilde{\omega}  \right ) &= \partial \bar{\partial} \left (  e^g \omega  \right ) = 0.
\end{align*}Moreover, since $X$ is a surface, so we automatically have $\partial \bar{\partial} \left ( \tilde{\omega}^2 \right ) = 0$. At $x \in X$, we have
\begin{align*}
 \Theta_{\tilde{\omega}} \left ( \ubar{\chi} \right ) (x) = \arctan(e^g(x) \lambda_1) + \arctan(e^g(x) \lambda_2), 
\end{align*}where $\lambda_1, \lambda_2$ are the eigenvalues of $\omega^{-1}\ubar{\chi}$ at $x$. By doing the first and second variation, one can check that
\begin{align*}
 \Theta_{\tilde{\omega}} \left ( \ubar{\chi} \right ) (x)  &\geq \min_{c \in \{1, m, M\}} \left \{ \arctan(  c   \lambda_1) + \arctan(  c  \lambda_2)    \right \} > 0,
\end{align*}where $m = \inf_X e^g$ and $M = \sup_X e^g$. As a consequence, by applying \hyperlink{T:4.1}{Theorem 4.1}, we can find $\chi \in [\chi_0]$ solving the dHYM equation 
\begin{align*}
\Theta_{\tilde{\omega}} \left ( \chi \right ) = \hat{\Theta}_{\tilde{\omega}} \left ( \ubar{\chi} \right ) = \hat{\Theta}_{\tilde{\omega}} \left ( \chi_0 \right ).
\end{align*}

\end{proof}

Moreover, we have the following result for some non-Kähler compact complex surfaces diffeomorphic to solvmanifolds.

\begin{fcor}
Let $X$ be either a Inoue surfaces or a secondary Kodaira surface with a Hermitian metric $\omega$. Then for any $[\chi_0] \in H^{1,1}_{\text{BC}}(X; \mathbb{R})$, there exists a pluriclosed Hermitian metric $\tilde{\omega}$ in the same conformal class of $\omega$ and a unique smooth $(1,1)$-form $\chi \in [\chi_0]$ solving the deformed Hermitian--Yang--Mills equation
\begin{align*}
\Theta_{\tilde{\omega}} \left ( \chi \right ) = \hat{\Theta}_{\tilde{\omega}} \left ( \chi \right ).
\end{align*} 
\end{fcor}

\begin{proof}
Since $X$ is a compact surface diffeomorphic to  a solvmanifold $\Gamma \backslash G$, it is endowed with a left-invariant complex structure.\bigskip

In each case, by taking a basis $\{e_1, e_2, e_3, e_4\}$ of the Lie algebra $\mathcal{G}$ associated to $G$, we have the following commutation relations (see \cite{hasegawa2005complex}). \\
{\em Inoue surface} of type $\mathcal{S}_M$. Differentiable structure:
\begin{align*}
[e_1, e_4] = -\alpha e_1 + \beta e_2,\quad [e_2, e_4] = -\beta e_1 - \alpha e_2, \quad [e_3, e_4] = 2\alpha e_3,
\end{align*}where $\alpha \in \mathbb{R}  \slash \{0\}$ and $\beta \in \mathbb{R}$.\\
{\em Inoue surface} of type $\mathcal{S}^{\pm}$. Differentiable structure:
\begin{align*}
[e_2, e_3] = - e_1,\quad [e_2, e_4] = -e_2, \quad [e_3, e_4] = e_3.
\end{align*}
{\em Secondary Kodaira} surface. Differentiable structure:
\begin{align*}
[e_1, e_2] = - e_3,\quad [e_1, e_4] = e_2, \quad [e_2, e_4] = -e_1.
\end{align*}
Denote by $\{ e^1, e^2, e^3, e^4 \}$ the dual basis of $\{ e_1, e_2, e_3, e_4 \}$. For {\em Inoue surface} of type $\mathcal{S}_M$ and {\em Secondary Kodaira} surface, we consider the $G$-left invariant almost-complex structure $J$ on $X$ defined by 
\begin{align*}
J e_1 \coloneqq e_2, \quad J e_2 \coloneqq - e_1, \quad J e_3 \coloneqq e_4, \quad J e_4 \coloneqq - e_3.
\end{align*}The $G$-left invariant $(1,0)$-forms are
\begin{align*}
\varphi^1 \coloneqq e^1 + \sqrt{-1} e^2, \quad \varphi^2 \coloneqq e^3 + \sqrt{-1} e^4.
\end{align*}\bigskip

For {\em Inoue surface} of type $\mathcal{S}^{\pm}$, we consider the $G$-left invariant almost-complex structure $J$ on $X$ defined by 
\begin{align*}
J e_1 \coloneqq e_2, \quad J e_2 \coloneqq - e_1, \quad J e_3 \coloneqq e_4 - q e_2, \quad J e_4 \coloneqq - e_3 - q e_1,
\end{align*}where $q \in \mathbb{R}$. The $G$-left invariant $(1,0)$-forms are
\begin{align*}
\varphi^1 \coloneqq e^1 + \sqrt{-1} e^2 + \sqrt{-1} q e^4, \quad \varphi^2 \coloneqq e^3 + \sqrt{-1} e^4.
\end{align*}

By Angella--Dloussky--Tomassini \cite{angella2016bott}, we have
\begin{center}
\begin{tabular}{llr}  
\cmidrule(r){1-3}
Non-Kähler compact complex surface    & $H^{1,1}_{\text{BC}}$ & $\dim_{\mathbb{C}} H^{1,1}_{\text{BC}}$ \\
\midrule
Inoue surface of type $\mathcal{S}_M$      & $\left \langle  \varphi^{2} \wedge \bar{\varphi}^2  \right \rangle$    & 1      \\
Inoue surface of type $\mathcal{S}^{\pm}$        & $\left \langle  \varphi^{2} \wedge \bar{\varphi}^2  \right \rangle$     & 1      \\
Secondary Kodaira       & $\left \langle  \varphi^{1} \wedge \bar{\varphi}^1  \right \rangle$     & 1      \\
\bottomrule
\end{tabular}
\end{center}

We prove the case when $X$ is a {\em Inoue surface} of type $\mathcal{S}_M$; the other cases are also similar. Since $\dim_{\mathbb{C}} H^{1,1}_{\text{BC}}(X) = 1$, we see that $[\chi_0] = c \cdot \sqrt{-1} \left [ \varphi^2 \wedge \bar{\varphi}^2 \right ]$, where $c \in \mathbb{R}$. If $c = 0$, then by Gauduchon \cite{gauduchon1977theoreme}, there exists a function $g \colon X \rightarrow \mathbb{R}$ such that the Hermitian metric $\tilde{\omega}\coloneqq e^g \omega$ satisfies 
\begin{align*}
\partial \bar{\partial} \left (   \tilde{\omega}  \right ) &= \partial \bar{\partial} \left (  e^g \omega  \right ) = 0.
\end{align*}Thus, $\Theta_{\tilde{\omega}} \left ( 0 \right ) = 0 = \hat{\Theta}_{\tilde{\omega}} \left (  \chi_0  \right )$, which solves the deformed Hermitian--Yang--Mills equation. The case $c < 0$ and the case $c > 0$ only differ by a sign. Let us consider the case $c > 0$. We claim that $c \sqrt{-1} \varphi^2 \wedge \bar{\varphi}^2$ is a $C$-subsolution. Let $\Lambda = \omega^{-1} \left (  c \sqrt{-1} \varphi^2 \wedge \bar{\varphi}^2  \right )$ and $\{ \lambda_1, \lambda_2 \}$ be the eigenvalues of $\Lambda$ with $\lambda_1 \geq \lambda_2$. Now, we wish to apply \hyperlink{C:4.1}{Corollary 4.1}, so we try to show that
\begin{align*}
\Theta_\omega \left (  c \sqrt{-1} \varphi^2 \wedge \bar{\varphi}^2  \right ) = \arctan{\lambda_1} + \arctan{\lambda_2} > 0.
\end{align*}This is equivalent to $\Tr_{\omega} \left ( c \sqrt{-1} \varphi^2 \wedge \bar{\varphi}^2 \right ) = \lambda_1 + \lambda_2 > 0$ when the complex dimension equals $2$. We have
\begin{align*}
\Tr_\omega \left (  c \sqrt{-1} \varphi^2 \wedge \bar{\varphi}^2  \right )  &=  \frac{2 c \sqrt{-1} \varphi^2 \wedge \bar{\varphi}^2 \wedge \omega}{\omega^2} = \frac{-2 c \cdot \omega_{1 \bar{1}} \varphi^{1 \bar{1} 2 \bar{2}}}{-2 \left ( \omega_{1 \bar{1}} \omega_{2 \bar{2}}  - \omega_{1 \bar{2}} \omega_{2 \bar{1}}  \right ) \varphi^{1 \bar{1} 2 \bar{2}}} \\
&=  c \frac{ \omega_{1 \bar{1}} }{ \left ( \omega_{1 \bar{1}} \omega_{2 \bar{2}}  - \omega_{1 \bar{2}} \omega_{2 \bar{1}}  \right ) },
\end{align*}where $\varphi^{1 \bar{1} 2 \bar{2}} = \varphi^1 \wedge \bar{\varphi}^1 \wedge \varphi^2 \wedge \bar{\varphi}^2$ for convenience. Here, we use the $G$-left invariant $(1, 0)$-forms $\varphi^1, \varphi^2$ and write $\omega = \sqrt{-1} \omega_{i \bar{j}} \varphi^i \wedge \bar{\varphi}^j$. Since $\omega$ is a Hermitian metric, we have
\begin{align*}
\omega_{1 \bar{1}} > 0; \quad  \omega_{1 \bar{1}} \omega_{2 \bar{2}}  - \omega_{1 \bar{2}} \omega_{2 \bar{1}}   > 0,
\end{align*}which completes the proof.
\end{proof}



\Address

\end{document}